\documentclass[a4paper,10pt,reqno]{amsart}
\usepackage[T1]{fontenc}
\usepackage[utf8]{inputenc}
\usepackage{comment}
\usepackage{stmaryrd}
\usepackage{mparhack}
\usepackage{amsmath,amssymb,amsthm,mathrsfs,eucal}
\usepackage{booktabs}
\usepackage{graphicx,subfig}
 \usepackage{graphics} 
  \usepackage{epsfig} 
\usepackage{graphicx}  \usepackage{epstopdf}
\usepackage{wrapfig}
\usepackage[bookmarks=true,colorlinks=true]{hyperref}
\usepackage{bm}

\usepackage{color}

\graphicspath{{nuova/}}
\usepackage{color}
\definecolor{commenti}{rgb}{0.13,0.55,0.13}
\definecolor{stringhe}{rgb}{0.63,0.125,0.94}

\usepackage{listings}

\newtheorem{defin}{Definition}[section]
\newtheorem{teo}{Theorem}[section]
\newtheorem{prop}{Proposition}[section]
\newtheorem{lem}{Lemma}[section]

\newtheorem{oss}{Remark}[section]
\newenvironment{proof-sketch}{\noindent{\em Sketch of the Proof.}\hspace*{1em}}{\qed\bigskip}

\newcommand{\R}{{\mathbb{R}}}

\newcommand{\mP}{{\mathscr{P}}}
\newcommand{\mF}{{\mathcal{F}}}
\newcommand{\ep}{{\epsilon}}

\title[Sorting Phenomena For Two Mutually Attracting/Repelling Species]
{Sorting Phenomena in a Mathematical Model For Two Mutually
Attracting/Repelling Species}

\author[M. Burger, M. Di Francesco, S. Fagioli, A. Stevens]{Martin Burger \and Marco Di Francesco \and Simone Fagioli \and Angela Stevens}

\address{Martin Burger, Angela Stevens -  Angewandte Mathe\-ma\-tik,
Westf\"alische Wil\-helms-Universit\"at M\"unster, Einsteinstr. 62, 48149 M\"unster, Germany.}
\address{Marco Di Francesco, Simone Fagioli - Dipartimento di Ingegneria e
Scienze dell'In\-for\-ma\-zione e Matematica, Universit\`{a} degli Studi dell'Aquila, Via Vetoio 1, 67100 Coppito, L'Aquila, Italy.}
\date{}

\begin{document}

\maketitle

\begin{abstract}
Macroscopic models for systems involving diffusion, short-range repulsion, and
long-range attraction
have been studied extensively in the last decades.
In this paper we extend the analysis to a system for two species
interacting with each other according to different inner- and
intra-species attractions. Under suitable conditions on this
self- and crosswise attraction an interesting effect can be observed,
namely phase separation into neighbouring regions, each of which
contains only
one of the species. We prove that the intersection of the support of the
stationary solutions of the continuum model for the two species has
zero Lebesgue measure, while the support of the sum of the two densities is
a connected interval.

Preliminary results indicate the existence of phase
separation, i.e. spatial sorting of the different species.
A detailed analysis is given in one spatial dimension. The existence and shape of segregated stationary
solutions is shown via the Krein-Rutman theorem. Moreover, for small
repulsion/nonlinear diffusion, also uniqueness of these stationary
states is proved.
\end{abstract}

\section{Introduction}

The interplay between (nonlinear) diffusion and nonlocal attractive/repulsive
interactions arises in a variety of contexts in the natural-, life-,
and social sciences.
As a non exhaustive list of examples, let us mention granular media physics
\cite{caglioti,brilliantov}, astrophysics \cite{biler, wolansky92},
semiconductors \cite{gajewski-groeger, wolansky92},
chemotaxis \cite{keller-segel, schaaf,
jaegerluckhaus, alt, post, gajewski-zacharias},
ecology, animal swarming and aggregation \cite{
mimura-yamaguti, nagai, alt,boi,capasso,topaz}, alignment
\cite{Primi, Geigant}, and opinion formation
\cite{sznajd,aletti}.
One way of deriving such models from
\emph{first order} microscopic systems of
(stochastic) interacting particles $x_1,\ldots,x_n$
is the following.
Each particle $x_i$ is driven by nonlinear forces due to short range
repulsion, see \cite{capasso} - respectively undergoes an independent
Brownian
motion. Further, it moves towards higher
concentrations of particles of its own kind, respectively those
of an external signal, \cite{Stevens2000}.
%
%
At the macroscopic and continuum level, this set of rules results in
a nonlocal partial differential
equation
\begin{equation}\label{eq:intro_one_species}
\partial_t \rho = \mathrm{div}\left[\rho\nabla\left( a(\rho)
- W\ast\rho\right)\right],
\end{equation}
respectively, a chemotaxis-type system
\begin{equation}
\partial_t \rho = \mathrm{div} \left[ \mu (\rho, v)  \nabla \rho
- \chi (\rho, v) \rho \nabla v \right] \quad , \quad
\tau \partial_t v = \eta \Delta v + k(\rho,v) \ .  \label{chemotaxis}
\end{equation}
For $\mu(\rho,v) = \rho a'(\rho)$, $\chi(\rho,v) = \chi_0$,
$\tau = 0$, and $k(\rho,v) = \rho - \beta v$, thus
$ v = (\beta I - \Delta)^{-1} \rho $,
system (\ref{chemotaxis}) is equivalent to (\ref{eq:intro_one_species}), with
$ W $ being the
Newtonian- or Bessel-potential ($\beta = 0$ or $\beta >0$).
Especially when $a(\rho) = \log (\rho)$, i.e.
$\mu(\rho,v)= 1$ in this case, then a
prototype model for nonlocal aggregation phenomena results,
namely the simplified, classical parabolic-elliptic Keller-Segel model
\cite{keller-segel},
or in physics a model for
gravitational
self-interacting clusters
\cite{wolansky92}.

In (\ref{eq:intro_one_species}),
$\rho=\rho(x,t)$ denotes the
density of particles,
$a:[0,\infty)\rightarrow[0,+\infty)$ is a $C^1$ monotone increasing
function with $a'(0)=0$, and $W(x)=\tilde{W}(|x|)$ is a $C^1$ potential
with $\tilde{W}'(r)< 0$ for all $r>0$.
The PDE \eqref{eq:intro_one_species} can be interpreted as the gradient flow of the functional
\begin{equation}
\mathcal{F}[\rho]=\int A(\rho)dx - \frac{1}{2}\int \rho W\ast \rho \,  dx
\,
\mbox{ where } A(\rho)=\int_0^\rho a(\xi)d\xi \, , \label{qqq}
\end{equation}
w.r.t. the $d_2$ Wasserstein distance arising in optimal
transport theory. Compare also the different Lyapunov functionals and energies used for \eqref{chemotaxis} in \cite{wolansky92, gajewski-zacharias, post}, e.g. $A(\rho)= \rho \log \rho$.
Roughly speaking, this means that the velocity
$\tilde v =- \nabla \left(a(\rho) - W\ast \rho\right)$
in the continuity equation
$\rho_t + \mathrm{div}\left(\rho \tilde v\right)=0$
can be interpreted as the sub-differential of $\mathcal{F}$ in the
metric space $\mathcal{P}_2$ of probability measures with finite second
moment and metric $d_2$, see \cite{AGS} for
more details. \\

Models of type
\eqref{eq:intro_one_species}, \eqref{chemotaxis} with \emph{two or more species}
being involved are considered in the context of chemotaxis
\cite{espejo,horst_lucia,
wolansky1, wolansky2}, opinion formation \cite{during}, pedestrian
dynamics \cite{degond}, and population biology
\cite{chen_kolokolnikov,dif_fag2}.
A reasonable
generalization of \eqref{qqq} for two species is
\begin{equation}
\mathcal{E}[\rho_1,\rho_2] = \int f(\rho_1,\rho_2) dx -
\frac{1}{2}\int \rho_1 S_1\ast \rho_1 dx -
\frac{1}{2}\int \rho_2 S_2\ast \rho_2 dx -
\int \rho_1 K\ast \rho_2 dx, \label{twospeciesfunctional}
\end{equation}
where $f:[0,+\infty)\times[0,+\infty)\rightarrow \R$, $f\in C^1$,
and $S_1$, $S_2$, $K$ are $C^1$ and radially decreasing potentials
like $W$ above.
The first term of $\mathcal{E}$ typically represents local repulsion, while
the nonlocal terms model attractive
forces. We call $S_1$ and $S_2$ \emph{self-interaction} potentials
and $K$ \emph{cross-interaction} potential.
For $(\mathcal{P}_2(\R^d),d_2)\times (\mathcal{P}_2(\R^d),d_2)$ with the
natural
product topology, the formal gradient flow of $E[\rho_1,\rho_2]$
w.r.t. this metric structure is given by
\begin{equation}\label{eq:intro_two_species_general}
\begin{cases}
\partial_t \rho_1 = \mbox{div}\big[\rho_1\nabla \big( f_{\rho_1}
(\rho_1,\rho_2) -S_1 \ast \rho_1 - K\ast \rho_2\big)\big] & \\
\partial_t \rho_2 = \mbox{div}\big[\rho_2\nabla \big( f_{\rho_2}
(\rho_1,\rho_2) -S_2 \ast \rho_2 - K\ast \rho_1\big)\big]\,. &
\end{cases}
\end{equation}
The functionals we investigate in this paper are of the form
(\ref{twospeciesfunctional}),
with
\[f(\rho_1,\rho_2)={\ep}(\rho_1+\rho_2)^2/2,\quad \epsilon>0,   \]
leading to
\begin{equation}
\begin{cases}
                 \partial_{t}\rho_1=\mbox{div}\big[\ep
\rho_1\nabla(\rho_1+\rho_2)-\rho_1\nabla S_{1}\ast \rho_1-\rho_1\nabla K\ast
\rho_2\big] \\
                \partial_{t}\rho_2=\mbox{div}\big[\ep \rho_2\nabla(\rho_1+\rho_2)-\rho_2\nabla S_{2}\ast \rho_2-\rho_2\nabla K\ast \rho_1\big]
\end{cases}.
\label{eq.diffusion_1}
\end{equation}
Besides modelling local repulsion and global attraction
between different types of particles (cf. the introduction and
references in \cite{franek})
our particular motivation is to consider
cell sorting due to differential attraction and the resulting
pattern formation.
Similar cells or species of equal size
(consequently the terms $(\rho_1+\rho_2)$ in
(\ref{eq.diffusion_1})), but with different reactions to attraction forces
undergo a reorganization process, where cells with stronger self-attraction
finally sort into the center of the total cell population and those with weaker
self-attraction to the outside.
Differential attraction can either be by an external (chemical)
signal or directly between the species.
Such phenomena are observed in developmental processes.
In the first case one may assume that $S_1$, $S_2$, and $K$ are
multiples of the same kernel, modelling indirectly a chemo-attractant;
like the fundamental solution of the elliptic equation for
the chemo-attractant in the prototype
Keller-Segel system, which can be rewritten into the single
prototype aggregation
equation.
We are interested in suitable conditions for the
self- and the
cross-attraction,
which result in the above mentioned sorting phenomenon.
%
%
In \cite{espejo} for differential attraction, i.e. different chemotactic
sensitivities,  of two species towards
higher concentrations of one chemo-attractant,
it was proved that if the solution for the more
strongly attracted species
blows up, than also the second one
blows up at the
same time. The amount of mass, which concentrates in
the joint blowup is different though, and
controlled by the system parameters.
This last, formal result, hints towards
a possible
separation of the main amount of masses of the two species.
The more strongly attracted species accumulates more mass
in the blowup than the other one.
Cell sorting due to differential adhesion/attraction was also discussed
in \cite{vossboehme, king},
where stochastic particle models and
numerical studies of continuum models
were considered. Based on experimental results, see e.g.
\cite{matsukuma1979chemotactic},
sorting
due to differential chemotaxis and chemical spiral waves was
modeled and simulated in \cite{vasiev-weijer, Painter2009}. In
\cite{kang-primi-velazquez} a rigorous analysis was given.
\\


We are interested in steady states of 
\eqref{eq.diffusion_1}
and minimizers of \eqref{twospeciesfunctional}, hence let us first discuss
some relevant results
in the single species setting \eqref{chemotaxis} and \eqref{eq:intro_one_species}.
Varying $\chi$ and $k$ in (\ref{chemotaxis}) plays a similar
role as varying $W$ in (\ref{eq:intro_one_species}) for
the pattern formation properties of the respective solution. Though a
one-to-one connection between all variants of aggregation
equations and chemotaxis systems as in the case of the
Newtonian- and Besselpotential for $(-W)$ and the prototype
Keller-Segel system has not been established, both
types of models and their dynamics are strongly intertwined.

A wealth of mathematical results on existence of global solutions,
blowup phenomena and pattern formation exists for these
models, see e.g. the summaries in \cite{horstmann1, horstmann2, perthame-book}.
The longtime behavior of the prototype models is especially
interesting in two dimensions as a biological model for the peculiar
chemotactic self-organisation of {\sl{Dictyostelium discoideum}}
\cite{jaegerluckhaus},
and for self-graviational
collapse and star formation in astrophysics in dimension three \cite{biler}.
In both cases, see also \cite{nagai}, blowup phenomena in the respective dimension, and point support
of stationary Dirac-type solutions are of crucial relevance for the respective
application.
Non-trivial stationary solutions and their qualitative
features have been analyzed in \cite{schaaf}
for the general system (\ref{chemotaxis})
with Neumann boundary conditions,
for $\chi (\rho, v) = \chi_0 \mu (\rho, v) \rho \phi' (v)$,
and $\tau = 1$.
Especially the one dimensional steady states and their stability are
well understood. These results should be compared with aggregation
equation analoga.
With $\rho = C \exp(\chi_0 \phi(v))$
the steady state analysis in \cite{schaaf} reduces to analyzing
the elliptic equation $\eta \Delta v + k(\psi (v, \lambda), v) = 0$,
where $\psi$ is the flux of $r'(s) = \chi(r,s)/\mu(r,s) = \chi_0 \rho
\phi'(v)$. 

In \cite{alt} density dependent diffusion-drift equations
for aggregation of the form
\begin{equation}
\partial_t \rho = \partial_x \left[ \mu (\rho) \partial_x \rho
- \gamma (\cdot, \rho) \right],  \label{alt-equation}
\end{equation}
are analyzed
with e.g. $\mu(\rho)= m \rho^{m-1}$ and
$\gamma (t,x,\rho)$ being proportional
to $\rho$, respectively depending on a functional of
the density
distribution $\rho(t,\cdot)$. Thus integral terms are included,
as discussed in \cite{mimura-yamaguti, nagai}, and (\ref{alt-equation})
can be compared with (\ref{eq:intro_one_species}).
The connections between chemotaxis systems and (\ref{alt-equation})
are given, and a Lyapunov function is constructed,
i.e. a functional decreasing along solutions
as time increases.  This provides a general theorem
on global existence of solutions. Asymptotic
convergence to steady states is proved, which
can be non-trivial, for instance plateau-like.

Free energies in such models and in chemotaxis
are based on an entropic term, e.g. $A(\rho) = \rho \log \rho - \rho + 1 $ (or just $A(\rho) = \rho \log \rho$) in \eqref{qqq}.
In \cite{wolansky92} not only
global minimizers of \eqref{qqq} are studied, but the whole solution set.
Here $\rho$ is a steady solution of the governing PDE - the associated gradient flow - if and only if it is
a critical point of the variational problem with constraint
$\int \rho = M$.
%
%
Due to the logarithmic term in the entropy we have $a(\rho)  \sim \log \rho$ and inverting this relation to an {\em entropy variable}
$\rho \sim \exp a(\rho)$  one can obtain simplified systems for stationary states.
With this exponential transformation and the Moser-Trudinger inequality
as given in
\cite{moser, moser2}
the critical mass $8\pi$
for graviational collapse is deduced.
Existence, uniqueness, stability and symmetry breaking of stationary
solutions are proved via the
precise connections between the free energy and the
respective Vlasov-Fokker-Planck (chemotaxis-like) equation.

In \cite{gajewski-zacharias}
a Lyapunov functional - c.f. the energy functional in (\ref{qqq}) - for the
prototype chemotaxis-model in two dimensions is provided, and
extended to more general equations in \cite{post}.  Also here
the connection to the exponential
transformation and the elliptic equation given in \cite{moser} is notified.
Geometric criteria for not necessarily trivial stationary states are
derived.
Now any metric $ds^2$ on a two-dimensional sphere determines a Gauss
curvature function $K$ satisfying the Gauss-Bonnet formula,
$ {\int\int}_{S^2} K d \tilde \mu = 4\pi$.
Here $ d\tilde \mu$ is the volume element of $S^2$.
To characterize all $K$ belonging to
metrices $ds^2$ and relating to the standard metric $ds_0^2$
via $ds^2 = p \,  ds_0^2$, with a positive function $p$
on the sphere, one has to determine $p = p (K)$
uniquely. In \cite{moser} it is proved that
the transformation $p = \exp(v)$ reduces this question
to solving
\begin{equation*}
\Delta v + K e^{2v} - 1 = 0,
\end{equation*}
on the sphere,
which is done by a variational approach.
This is a specific form of the
elliptic problem analyzed in \cite{schaaf}, where the
more general transformation
$\rho = C \exp(\chi_0 \phi(v))$ was used
for the steady states analysis of
generalized chemotaxis systems.
With $k(\psi(v, \lambda), v) = K e^{2v} -1 $, i.e.
$\psi (v,\lambda) = Ke^{2v}$, respectively $\chi_0 \phi = 2 v$
in \cite{schaaf}, the above elliptic PDE results.
Qualitative features of simplified chemotaxis systems with non-linear
diffusion have been discussed in \cite{sugiyama1, sugiyama2}.
See also further references therein.
\\

Existence and uniqueness
for \eqref{eq:intro_one_species} are discussed e.g. in
\cite{carrillo_mccann_villani}, the
last part of the book \cite{AGS}, and earlier
via the related works on chemotaxis, c.f.
\cite{jaegerluckhaus, velazquez, horstmann1, horstmann2, sugiyama1, sugiyama2}.
In \cite{bertozzi_et_al} also
singular potentials $W$ are considered, e. g. Coulomb
potential
as in electrodynamics, $W(x) = \vert x \vert$, $W(x) = \vert x \vert^\alpha$,
$\alpha > 2 - d $. We also refer to
\cite{perthame,blanchet_carrillo_carlen,blanchet_carrillo_laurencot} for further discussion.

%


For $a(\rho)=\log\rho$ in (\ref{qqq}), without any nonlocal
effects, but with the addition of a confinement external potential $V$,
%
%
in \cite{JKO} weak solutions to the linear Fokker-Planck
equation in a gradient flow setting
were derived
by constructing the finite time-step minimizing movement
\[\rho^n\mapsto \rho^{n+1}=\mathrm{argmin}_\rho\left(F[\rho]
+\frac{1}{2\Delta t} d_2^2(\rho^n,\rho)\right), \]
and taking the limit $\Delta t\rightarrow 0$. This idea was applied to
\eqref{eq:intro_one_species} in \cite{carrillo_mccann_villani}, and to a
more general metric framework in \cite{AGS}.

%
%
%

The large time
asymptotics of \eqref{eq:intro_one_species} depend on the competing
effects of $\int A(\rho)$, which promotes particle spreading,
such that the density $\rho $ stabilises at a constant state
(which is zero if we consider $\R^d$), whereas the nonlocal term drives the particles
towards aggregation. A Dirac delta results, which is located at the
(preserved) center of mass
of the system. To prove existence of a global minimum of
$\mathcal{F}$ is therefore a challenging problem. In \cite{bedrossian} this issue
was tackled in detail by using Lions' concentrated compactness technique.

The limiting case $a=0$ was analyzed in
\cite{xxx2007, choi_carrillo_hauray}, see also the references therein.
The case $a(\rho)=\rho$ was derived as a nonlocal
\emph{repulsive} effect under the action of a potential
$V_\epsilon$ which converges weakly to a Dirac delta as
$\epsilon \rightarrow 0$. For a formal argument see \cite{capasso},
see also \cite{oelschlaeger} for a rigorous derivation via interacting particle systems with short-repulsion (the repulsion range shrinking to zero as the number of particles goes to infinity). The results in \cite{franek}
and \cite{bedrossian} provide a clear picture of the
behaviour of the corresponding functional
\[\mathcal{G}[\rho]=\frac{1}{2}\int \rho \big(\epsilon\rho + W\ast
\rho\big) dx,\]
with a given diffusion constant $\epsilon>0$, and
a $W^{1,1}$ attractive potential $W\leq 0$. If $\epsilon \geq\|W\|_{L^1}$, then
$\mathcal{G}$ is uniformly convex, and hence zero is its global minimizer.
This suggests that in this case
$\rho\rightarrow 0$ for large times.
On the other hand, if $\epsilon<\|W\|_{L^1}$, a nontrivial global minimizer
exists in the class $\int \rho \, dx = 1$. This suggests existence of
nontrivial
steady state for large times, but the solution to the Cauchy problem could still decay to zero if a large
enough variance is present initially. This problem is still open in the case
of slow diffusion, see \cite{bedrossian2014}. In one dimension a more refined
analysis can be given. In \cite{franek} it was proved that a unique steady
state with given mass and center of mass exists for
\begin{equation*}
\rho_t = [\rho (\epsilon \rho +W\ast \rho)_x]_x,
\end{equation*}
provided that $\epsilon<\|W\|_{L^1}$ and $W$ is radially increasing, negative,
and supported on the whole of $\R$. Such a steady state is also the unique global
minimizer for $\mathcal{G}$ for fixed mass and fixed center of mass, it is
compactly supported, symmetric, and in $W^{1,\infty}$, with a shape similar to
a Barenblatt profile for the porous medium equation, see \cite{vazquez}. Such steady state is locally stable for large times, see the recent \cite{yahya}. The result in \cite{franek} was partly extended to more general nonlinear diffusions in
\cite{burger_fetecau_huang}, where the existence of a unique diffusion
constant for a given support of the steady state is proven. Both results rely
on the Krein-Rutman theorem in order to characterise the steady states as
eigenvectors of a certain nonlocal operator. Further generalizations
and completions of some open questions in those papers were recently given
in \cite{Kaib, KaibPhD}.
One major issue solved in \cite{franek}
is to prove that a one-to-one correspondence exists between the diffusion constant
(eigenvalue) and the support of the steady state. The results in \cite{franek} improved a previous result in
\cite{BuDiF_NHM} about the existence and uniqueness of nontrivial steady
states for small diffusion constant $\epsilon\ll 1$ performed via an implicit function theorem argument.
\\

%
%

Now let us come back to the functional \eqref{twospeciesfunctional}.
The first term of $\mathcal{E}$ typically represents local repulsion.
%
%
Apart from nonlocal cross-diffusion via the kernels $S_i$ also $f$, when containing mixed terms in
$\rho_1$ and $\rho_2$,
introduces cross diffusion. Indeed, \eqref{eq:intro_two_species_general} can be formally rewritten as
\begin{equation}\label{eq:intro_two_species_general_rewritten}
\partial_t U = \mbox{div}\left(D(U)\nabla U\right) - \mbox{div}\left(\left(\begin{matrix} \rho_1 & 0 \\ 0 & \rho_2\end{matrix}\right) \nabla \left(\begin{matrix} S_1\ast \rho_1 + K\ast \rho_2 \\ S_2\ast \rho_2 + K\ast \rho_1\end{matrix}\right)\right),
\end{equation}
with $U=(\rho_1,\rho_2)$,
\[D(U)=\left(\begin{matrix} \rho_1 f_{\rho_1,\rho_1} & \rho_1 f_{\rho_1,\rho_2} \\ \rho_2 f_{\rho_1,\rho_2} & \rho_2 f_{\rho_2,\rho_2}\end{matrix}\right).\]
If $D(U)$ is symmetric and semi-positive definite, then the theory
in \cite{jungel_boundedness} applies, see also the recent \cite{esposito} in the context of Wasserstein gradient flows. However, in our situation $D(U)$ is never
symmetric.
This makes existence proofs for solutions of
\eqref{eq:intro_two_species_general} in full generality a delicate problem.
In most cases, $D(U)$ does not even have a semi-positive definite symmetric
part.

A comprehensive existence and uniqueness theory for
\eqref{eq:intro_two_species_general} with $f\equiv 0$ was given in
\cite{fagioli}, where convergence of the JKO scheme leads to existence of weak
measure valued solutions for `mildly singular' potentials. A suitable notion
of displacement convexity for systems provides a uniqueness result. An
implicit-explicit variation of the JKO scheme yields an existence result in
\cite{fagioli}, also for nonlocal terms which are not of gradient flow type,
i.e. for non symmetric cross-interaction terms. Without the nonlocal interaction part, an existence theory with a direct application of the JKO scheme is
given in \cite{laurencot} for a fully coupled system of two second-order
parabolic degenerate equations arising as a thin film approximation to the
Muskat problem. The peculiar form of the cross-diffusion in this case allows
for an energy functional with good coercivity properties. In \cite{laurencot}
the regularity needed in order to identify the suitable Euler-Lagrange
equation in the
variational scheme was obtained. Although not strictly related to our model,
let us also mention the hybrid variational scheme generalizing the JKO scheme.
This has been introduced in \cite{carrillo&co} for the parabolic-parabolic
Keller-Segel model in $\R^2$ working in the product space
$\mP_2(\R^2)\times L^2(\R^2)$. Using a modified Wasserstein distance between
vector-valued densities on $\R$, in \cite{matzin} the variational structure of
systems with degenerate diffusion and nonlinear reaction terms was investigated.
A more general approach has been discussed in \cite{mielke}, where a
convexity concept
for reaction-diffusion systems was developed,
that allows to analyze some important examples.
\\

In system (\ref{eq.diffusion_1}) the symmetric part of $D(U)$ is not
semi-positive definite for all $\rho_1,\rho_2\geq 0$ in general. A brute-force
JKO approach would imply a good dissipation estimate only for $w=\rho_1+\rho_2$, but not for both $\rho_1$ and $\rho_2$. Thus a
regularising effect for the sum $w$ is still possible. On the other hand, such a degenerate dissipation estimate does not prevent the formation of discontinuities for $\rho_1$ and $\rho_2$ separately.
In order to partly validate this hypothesis, we focus on stationary states
in one space dimension, and show that \emph{discontinuous stationary patterns}
may arise. More precisely, we prove that $\rho_1$ and $\rho_2$ can separate
completely in the stationary state, i.e. they feature a jump discontinuity
at exactly the same point, with the sum $w=\rho_1+\rho_2$ remaining smooth
at that point. We call such a configuration a \emph{fully segregated steady state}.

Segregation in multi-species systems with nonlinear
diffusion terms like in \eqref{eq.diffusion_1} and possible reaction terms
has been widely investigated in the literature. We refer to
\cite{bertsch_gurtin_hilhorst_peletier,bertsch_gurtin_hilhorst,
bertsch_hilhorst_izuhara_mimura,bertsch_dalpasso_mimura} and references
therein.
It is well known that segregated initial data produce a unique
segregated solution. The problem of existence of a mixed solution is partly
open. Difficulties in our case occur due to the presence of
the nonlocal terms, which typically do not allow for a local comparison
principle and may produce aggregation phenomena.
The emergence vs. non-emergence of segregated steady states in two species
systems with nonlocal attraction has recently been treated in
\cite{novaga_et_al}, where the repulsive effect of the nonlinear diffusion
has been replaced by an upper bound for $\rho_1+\rho_2$. Then the minimisation
problem for $\mathcal{E}$ with $f=0$ and $0\leq \rho_1+\rho_2\leq 1$ with
all interaction potentials being multiples of a given function $K$
is analyzed. Sharp conditions on these multiplying factors are provided,
which results in complete segregation of the
two species. In \cite{magni} rearrangements for three already separated
domains has been considered in a different context.

More specifically on system \eqref{eq.diffusion_1} (almost parallel to our result), the recent \cite{zoology} shows how to construct explicit segregated and non-segregated steady states with power-law interaction potentials (with possible confinement effects) and produces numerical evidence that segregation may not occur for large $\epsilon$. At the same time, \cite{alpar} shows that initially segregated initial data yield segregated solutions for all times for a system with the same (cross-)diffusion term of \eqref{eq.diffusion_1} and external potentials (under suitable assumptions on the latter). We also mention at this stage the result in \cite{CFSS}, in which an existence theory for a reaction diffusion system with the same (cross-)diffusion of \eqref{eq.diffusion_1} has been proven via JKO scheme with $BV$ initial conditions, without any initial separation assumption.

Let us know briefly summarise our results. For the one-dimensional case, we first investigate
conditions for existence or non-existence of non trivial stationary solutions.
For small diffusion coefficient $\ep$ we
show existence of segregated stationary states via the
implicit function theorem. We extend the strategies in
\cite{budif,franek} to the multi species case. We also relate stationary solutions to energy minimizers in a rigorous way and provide some results characterizing their structure. In particular we verify that for stationary solutions and energy minimizers, the sum $w$ is supported on a connected interval and we give a rigorous result on the segregation in the case of dominant self-attraction. In a case of weak self-attraction we can characterized the support of the minimizers by the same arguments as in \cite{novaga_et_al}.

For the interaction kernels in \eqref{eq.diffusion_1}
we assume unless further noted:
\begin{itemize}
\item [(A1)] $S_1, S_2, K \in C^2(\R)$,
\item [(A2)] $S_1, S_2, K$ are radially symmetric and decreasing w.r.t.
the radial variable,
\item [(A3)] $S_1, S_2, K$ are nonnegative and have finite mass on $\R$.
\end{itemize}

The paper is organized as follows. In Section 2 we give
some preliminary results on phase separation for \eqref{eq.diffusion_1}. In Section 3 we analyze the
relation of stationary solutions with critical points of the associated energy
functional and provide sufficient conditions on $\epsilon$ and on the
interaction kernels yielding non existence of non-trivial steady states.
Section 4 is devoted to the existence analysis of stationary solutions
involving phase-separation via the Krein-Rutman based approach given in
\cite{franek}. Finally, in Section 5 existence and uniqueness results
for stationary solutions are proved in case of small repulsion,
corresponding to small nonlinear diffusion in \eqref{eq.diffusion_1}. We complement our results with some numerical simulations in Section 6 showing segregated behavior as well as mixing and diffusion-dominated behaviour for large diffusions.

%

\section{Segregation due to differential aggregation}

First, we provide some preliminary results in arbitrary dimensions about
pattern formation for model \eqref{eq.diffusion_1}, and more specifically on
the emergence or non-existence of segregation.
%
%
Consider the canonical model
\begin{equation}
\label{eq.confinement1}
\begin{cases}
\partial_{t}\rho_1=\mbox{div}\big(\ep \rho_1\nabla(\rho_1+\rho_2)
-\rho_1 \nabla V_1\big), \\
		\partial_{t}\rho_2=\mbox{div}\big(\ep \rho_2\nabla(\rho_1+\rho_2)-\rho_2 \nabla V_2 \big)\, ,
\end{cases}
\end{equation}
where $V_1$ and $V_2$ are two given smooth external potentials. We can
rewrite \eqref{eq.diffusion_1} in this form, with
$V_1$ and $V_2$ being determined by convolutions of $\rho_1$ and
$\rho_2$ with aggregation kernels, c.f. (\ref{eq:V1V2}).

\begin{prop} \label{prop.canonicalsegregation}
Let $V_i \in L^1_{loc}(\R^d) \cap W^{1,\infty}(\R^d)$ be given external potentials.
If $(\rho_1^\infty,\rho_2^\infty)$ is a $C^1$ (weak) stationary solution
of \eqref{eq.confinement1}, then we have
\begin{equation}
\mbox{supp}(\rho_1^\infty) \cap \mbox{supp}(\rho_2^\infty) \subseteq \{\nabla V_1 = \nabla V_2\}.
\end{equation}
\end{prop}
\begin{proof}
Let $x \in \mbox{supp}(\rho_1^\infty) \cap \mbox{supp}(\rho_2^\infty)$. Then
the weak formulation of \eqref{eq.confinement1} implies
$$ \ep \nabla(\rho_1^\infty+\rho_2^\infty)(x) = \nabla V_1(x) = \nabla V_2(x), $$
from which the assertion follows.
\end{proof}

An immediate consequence of Proposition \ref{prop.canonicalsegregation}
is the following result, which deals with the special case of all the interaction kernels being multiples of the fundamental solution of the Laplace equation.
\begin{prop} \label{prop_laplace_segregation}
Let $K$ be the fundamental solution of the Laplace equation in $\R^d$, and
$S_1=\sigma_1 K$, $S_2=\sigma_2 K$ with $\sigma_1\leq 1\leq \sigma_2$ and
$\sigma_1\neq \sigma_2$. Then, every $C^1$ stationary solution $(\rho_1,\rho_2)$ of \eqref{eq.diffusion_1} is fully segregated, i. e. $\mbox{supp}(\rho_1)\cap \mbox{supp}(\rho_2)$ has empty interior.
\end{prop}

\begin{proof}
\begin{equation} \label{eq:V1V2}
\mbox{Let }\, V_1 = S_1*\rho_1+K*\rho_2, \qquad V_2=S_2*\rho_2 + K*\rho_1.
\end{equation}
Assume that there is a non-empty open set $\mathcal{O} \subset
\mbox{supp}(\rho_1)\cap \mbox{supp}(\rho_2)$. Then Proposition
\ref{prop.canonicalsegregation} implies $\mathcal{O}\subset\{\nabla V_1 =\nabla V_2\}$.
Hence  $V_1-V_2 = c = const.$ on $\mathcal{O}$, i.e.
$ (S_1 - K)* \rho_1 - (S_2-K)*\rho_2 = c$.
With the assumptions on the kernels, we can apply the Laplace operator
and obtain
$ (\sigma_1 -1)\rho_1(x) + (1-\sigma_2) \rho_2(x) = 0 $,
for all $x\in \mathcal{O}$. The assumptions on $\sigma_1$ and $\sigma_2$ imply $\rho_1(x)=\rho_2(x)=0$, which is a contradiction.
\end{proof}

Proposition \ref{prop_laplace_segregation} shows segregation of steady states
in a particular case. However, this effect occurs for a much wider
class of aggregation kernels $S_1, S_2, K$, as we will argue and partially prove below.

First, take a closer look at the dynamics of segregation by using a
transformation of variables similar to the one in
\cite{bertsch_hilhorst_izuhara_mimura}, where strong reaction terms induced
the segregation though. Let
\begin{equation}
w:= \rho_1 + \rho_2, \qquad \zeta:= \frac{\rho_1-\rho_2}{w},
\end{equation}
where the relative difference $\zeta$ is only considered on the support of
the total density $w$.
%
%
Adding both equations in \eqref{eq.confinement1} and using
the notion of (\ref{eq:V1V2})  yields
\begin{eqnarray*}	
\partial_{t}w&=&\mbox{div}\Big(\ep w \nabla w-\rho_1\nabla V_1 -
\rho_2\nabla V_2 \Big),  \\
&=&\mbox{div}\left(\ep w \nabla w-w\frac{1+\zeta}2\nabla V_1 - w\frac{1- \zeta}2
\nabla V_2 \right) \, .
\end{eqnarray*}
Thus the dynamics of $w$ are governed by a porous medium equation with
additional convective terms.
The dynamics of $\zeta$ are obtained by subtracting the
equations for $\rho_1$ and $\rho_2$ and then inserting the equation for $w$, which yields
\begin{equation*}
w	\partial_t \zeta = \left(\ep w \nabla w-w\frac{1+\zeta}2\nabla V_1
- w\frac{1- \zeta}2 \nabla V_2 \right)\cdot \nabla\zeta - \frac{1-\zeta^2}{2}
\mbox{div}\big(w \nabla (V_1 - V_2)\big).
\end{equation*}
As in \cite{bertsch_hilhorst_izuhara_mimura} the evolution of
$\zeta$ is governed by a first-order equation, which gives particular
insight
into the dynamics of the system. The first of the two terms on the
right-hand side of the above equation is along the flux of
$w$, which determines the spatial position and shape of the solution.
The crucial part for the segregation dynamics is the second term, a reaction
term w.r.t. $\zeta$ and two fixed points $\zeta=\pm 1$, corresponding to
segregation. Depending on the sign of
$\mbox{div}\left(w \nabla (V_1 - V_2)\right)$ one of them is stable.
Thus there is always some dynamics towards a segregated state driven
by the differences in the attraction forces.
Instead of pursuing a
time-dependent analysis inspired by the above considerations, we
restrict our search for segregated states to the analysis of stationary
solutions, and leave the stability vs. instability analysis of segregated
patterns to future work.

\section{Steady states vs energy minimization}\label{sec:minimisation}

We now explore the relation between one dimensional steady states of
\eqref{eq.diffusion_1}, namely solutions $(\rho_1,\rho_2)$ of
\begin{equation}\label{eq.stationary}
\begin{cases}
		 0=\big(\ep \rho_1(\rho_1+\rho_2)_x-\rho_1 S'_{1}\ast
\rho_1-\rho_1 K'\ast \rho_2\big)_x \\
	       0=\big(\ep \rho_2(\rho_1+\rho_2)_x-\rho_2S'_{2}\ast
\rho_2-\rho_2 K'\ast \rho_1\big)_x,
\end{cases}
\end{equation}
and the minimisers of the energy functional
\[\mF[\rho_1,\rho_2] = \frac{\epsilon}{2}\int(\rho_1+\rho_2)^2 dx - \frac{1}{2}\int \rho_1 S_1\ast \rho_1 dx - \frac{1}{2}\int \rho_2 S_2\ast \rho_2 dx - \int \rho_1 K\ast \rho_2 dx.\]
Note, that $\rho_1,\rho_2 \in L^2(\R)$ is a natural setting for the
minimization of $\mF$, since assumption (A3) ensures that the
nonlocal terms in $\mF$ are finite. More precisely, we will look for
minimizers
within the set
\begin{equation}
\mathcal{M} = \left \{(\rho_1,\rho_2) \in (L^2(\R)\cap L^1(\R))^2~\Big|~\rho_i
\geq 0, \,  \int_\R \rho_i~dx = m_i \,,\, i=1,2 \right \},\end{equation}
for given masses $m_1$, $m_2$.
The functional setting for the weak solutions of \eqref{eq.stationary} should
involve some space derivative, since the cross-diffusion term in
\eqref{eq.stationary} is not of type $\Delta F(\rho_1,\rho_2)$ with some
nonlinear vector field $F$. Therefore one cannot integrate by parts twice
in the distributional formulation. For a weak formulation involving spatial
derivatives of $\rho_1$ and $\rho_2$ and allowing for discontinuities of each
species at the same time, we choose $BV(\R)$ as the functional setting for the
steady states, whose sum $w=\rho_1+\rho_2$ is Lipschitz continuous. Such a choice is partly supported by the results in \cite{CFSS}, in which the preservation of the $BV$ regularity is proven for a similar system.

\begin{defin}\label{def:stationary}
The pair $(\rho_1,\rho_2)\in \mathcal{M}\cap BV(\R)^2$ is a \emph{weak solution} to \eqref{eq.stationary} if $w=\rho_1+\rho_2 \in \mathrm{Lip}(\R)$, and
\begin{align*}
& 0 = \int \rho_1 \Big( \ep(\rho_1+\rho_2)_x - S_1'\ast \rho_1 -K'\ast
\rho_2 \Big)U_x \,  dx, \\
& 0 = \int \rho_2 \Big( \ep(\rho_1+\rho_2)_x - S_2'\ast \rho_2 -K'\ast
\rho_1 \Big)V_x \,  dx,
\end{align*}
for arbitrary $U,V\in C^1_c(\R)$.
\end{defin}

\begin{prop}\label{prop:stationary}
Let $\left(\rho_1,\rho_2\right)\in \mathcal{M}\cap BV(\R)^2$ be a minimizer of
\[\mF[\rho_1,\rho_2]=\frac{\ep}{2}\int_{\R}(\rho_1+\rho_2)^2dx-\frac{1}{2}\int_{\R}\rho_1S_{1}*\rho_2dx-
\frac{1}{2}\int_{\R}\rho_2S_{2}*\rho_2dx-\int_{\R}\rho_1K*\rho_2dx,\]
such that $w=\rho_1+\rho_2\in \mathrm{Lip}(\R)$. Then $(\rho_1,\rho_2)$ is a
weak solutions of \eqref{eq.stationary} according to Definition
\ref{def:stationary}. Moreover, every weak solution $(\rho_1,\rho_2)$ of
\eqref{eq.stationary} according to Definition \ref{def:stationary} is a
critical point of $\mF$.
\end{prop}

\proof Let $(\rho_1,\rho_2)\in \mathcal{M}\cap BV(\R)^2$ be a minimizer of
$\mF$. We calculate the first Gateaux derivative of  $\mF$
\begin{equation*}
\frac{d}{d\rho_1}\mF[\rho_1,\rho_2](\mu)=\lim_{\delta\rightarrow 0}
\frac{1}{\delta}\Big(\mF[\rho_1+\delta \mu,\rho_2]-\mF[\rho_1,\rho_2]\Big),
\end{equation*}
along an arbitrary direction $\mu \in L^2(\R)$, such that $(\rho_1+\delta \mu,\rho_2)\in \mathcal{M}$. We have
\begin{align*}
\mF[\rho_1+\delta \mu,\rho_2]-\mF[\rho_1,\rho_2]
&  =\frac{\ep}{2}\int_{\R}(\rho_1+\delta \mu+\rho_2)^2dx-\frac{\ep}{2}\int_{\R}(\rho_1+\rho_2)^2dx\\
& \quad -\frac{1}{2}\int_{\R}(\rho_1+\delta \mu)S_{1}*(\rho_1+\delta \mu)dx+\frac{1}{2}\int_{\R}\rho_1S_{1}*\rho_1dx\\
&\quad  -\int_{\R}(\rho_1+\delta \mu)K*\rho_2dx+\int_{\R}\rho_1K*\rho_2dx\\
&  =\frac{\ep}{2}\int_{\R}(\delta \mu)^2+2\delta \mu (\rho_1+\rho_2)dx-\frac{1}{2}\int_{\R}2\delta \mu S_{1}*\rho_1dx\\
&  \quad -\frac{1}{2}\int_{\R}\delta^{2}\mu S_{1}*\mu dx-\int_{\R}\delta \mu K*\rho_2dx.
\end{align*}
Dividing by $\delta$ and taking the limit $\delta\rightarrow 0$ results in
\begin{equation*}
\frac{d}{d\rho_1}\mF[\rho_1,\rho_2](\mu)=\int_{\R}\mu
\Big(\ep(\rho_1+\rho_2)-S_{1}*\rho_1-K*\rho_2\Big)dx.
\end{equation*}
Let $U\in C^1_{c}(\R)$ be an arbitrary vector field and let
$\mu_\gamma=\partial_{x}(\rho_{1,\gamma} \, U_{\gamma,x})$, where the subscript
$\gamma$ denotes convolution with a standard (compactly supported)
$C^\infty$-mollifier. Then
\begin{equation*}
\frac{d}{d\rho_1}\mF[\rho_1,\rho_2]=-\int_{\R}\rho_{1,\gamma}\partial_{x}
\Big(\ep(\rho_1+\rho_2)-S_{1}*\rho_1-K*\rho_2\Big)\cdot
U_{\gamma,x} \, dx \, ,
\end{equation*}
is well defined, since $w$ is Lipschitz continuous. For
$\gamma\searrow 0$ one obtains the first equation of \eqref{eq.stationary} in
weak form. The second equation follows similarly.

In order to prove the last assertion in the statement, we work by
contradiction. Let $\rho_1$ and $\rho_2$ be as assumed and let a direction
$\mu\in L^2$, such that $(\rho_1+\delta \mu,\rho_2)\in \mathcal{M}$ for which, e.g.
\[\frac{d}{d\rho_1}\mF[\rho_1,\rho_2](\mu)=\int_{\R}\mu
\Big(\ep(\rho_1+\rho_2)-S_{1}*\rho_1-K*\rho_2\Big)dx \neq 0.\]
By a density argument, one finds a function $U\in C^2_c(\R)$ such that
\[\frac{d}{d\rho_1}\mF[\rho_1,\rho_2]((\rho_1 U_x)_x)\neq 0,
\mbox{ which implies } \]
\[0\neq \int \rho_1 U_x \Big(\ep(\rho_1+\rho_2)-S_{1}*\rho_1-K*\rho_2\Big)_x dx,\]
which is well defined since $\rho_1+\rho_2$ is Lipschitz. Thus $(\rho_1,\rho_2)$ cannot be a steady state.
\endproof

Next, introduce a technical result that will be useful to investigate sufficient conditions for a steady state to be a local
minimizer of $\mathcal{F}$ in the next subsection.

\begin{lem}\label{lem:second_gateaux}
The second Gateaux derivatives for $\mF$ on $(\rho_1,\rho_2)$
\[\frac{d^2\mathcal{F}}{d \rho_i d\rho_j}(\rho_1,\rho_2)[\mu,\nu]
= \lim_{\delta\rightarrow 0}\frac{1}{\delta}
\left(\frac{d\mathcal{F}}{d\rho_i}(\rho_1+ \delta_{1,j}\mu,\rho_2+\delta_{2,j}
\nu)-\frac{d\mathcal{F}}{d\rho_i}(\rho_1,\rho_2)\right),\footnote{$\delta_{i,j}$ denotes the Kronecker symbol.}\]
are given by
\begin{eqnarray*}
H[\mu,\nu]&=&\begin{pmatrix} \frac{d^2\mathcal{F}}{d \rho_1^2}(\rho_1,\rho_2)[\mu,\nu] & \frac{d^2\mathcal{F}}{d \rho_1 d\rho_2}(\rho_1,\rho_2)[\mu,\nu] \\
\frac{d^2\mathcal{F}}{d \rho_1 d\rho_2}(\rho_1,\rho_2)[\mu,\nu] & \frac{d^2\mathcal{F}}{d \rho_2^2}(\rho_1,\rho_2)[\mu,\nu]\end{pmatrix} \\
& =& \begin{pmatrix}\int_{\R}(\ep\mu^{2}-\mu S_{1}*\mu)dx & \int_{\R}(\ep\mu\nu-\mu K*\nu)dx \\
& \\
\int_{\R}(\ep\mu\nu-\mu K*\nu)dx & \int_{\R}(\ep\nu^{2}-\nu S_{2}*\nu)dx \end{pmatrix},
\end{eqnarray*}
where $\mu,\nu \in L^2(\R)^2$ are arbitrary and such that $(\rho_1+\delta \mu,\rho_2)$ and  $(\rho_1,\rho_2+\delta\nu)$ are in $\mathcal{M}$.
\end{lem}
\proof Computing the upper left element of $H[\mu,\nu]$ gives
\begin{eqnarray*}
&& \frac{d}{d\rho_1}\mF[\rho_1+\delta \mu,\rho_2]-\frac{d}{d\rho_1}\mF[\rho_1,\rho_2]\\
&& \quad \quad =\int_{\R}\mu \Big(\ep(\rho_1+\delta \mu+\rho_2)-S_{1}
*(\rho_1+\delta \mu)-K*\rho_2\Big)dx\\
&& \quad \quad \quad  -\int_{\R}\mu \Big(\ep(\rho_1+\rho_2)-S_{1}*\rho_1-K*
\rho_2\Big)dx\\
&& \quad \quad =\delta\int_{\R} \ep\mu^{2}-\mu S_{1}*\mu \, dx.
\end{eqnarray*}
Therefore, we obtain
\begin{equation*}
\frac{d^2}{d\rho_1^2}\mF[\rho_1,\rho_2](\mu)=\int_{\R}(\ep\mu^{2}-\mu S_{1}*\mu)dx.
\end{equation*}
The other entries of $H[\mu,\nu]$ can be computed similarly.
\endproof

\subsection{Non-Existence of Steady States}

We now establish a simple necessary condition for the existence of non trivial
steady states, based on the idea that $\mF$ is strictly
convex when the diffusion part is dominant. Thus zero is the unique global
minimizer.

\begin{lem}\label{lem:coercivity}
Assume that the Fourier transforms of the interaction kernels satisfy
\begin{equation}\label{eq:condition_coercivity}
(i) \quad  \ep> \max\{\hat{S}_1(\xi),\hat{S}_2(\xi)\}, \, \hbox{ and } \,
(ii) \quad (\ep-\hat{S}_1(\xi))(\ep-\hat{S}_2(\xi))> (\ep -\hat{K}(\xi))^2,
\end{equation}
for all $\xi \in \R$. Then there does not exist any global minimiser
$(\rho_1,\rho_2)\in \mathcal{M}$.
\end{lem}

\proof
Thanks to assumption (A3), there exists $C>0$ such that
\begin{equation}\label{eq:functional_bounded}
\mF[\rho_1,\rho_2]\leq C\left(\|\rho_1\|_{L^2}^2 + \|\rho_2\|_{L^2}^2\right).
\end{equation}
Applying the Fourier transform we get
\begin{eqnarray*}
\mF[\rho_1,\rho_2] &=& \frac{1}{2}\int\left(\ep -\hat{S}_1(\xi)\right)
\hat{\rho}_1^2(\xi) d\xi +
\frac{1}{2}\int\left(\ep -\hat{S}_2(\xi)\right) \hat{\rho}_2^2(\xi) d\xi  \\
&&  + \int \left(\ep-\hat{K}(\xi)\right)\hat{\rho}_1(\xi)\hat{\rho}_2(\xi)d\xi\\
& = & \int (\hat{\rho}_1(\xi),\hat{\rho}_2(\xi))^T \cdot A(\xi) \cdot (\hat{\rho}_1(\xi),\hat{\rho}_2(\xi)) d\xi,
\end{eqnarray*}
\[ \mbox{with } \,
A(\xi)=\left(\begin{matrix}  \frac{1}{2}(\ep -\hat{S}_1(\xi)) &  \frac{1}{2}(\ep-\hat{K}(\xi))\\  \frac{1}{2}(\ep-\hat{K}(\xi)) & \frac{1}{2}(\ep -\hat{S}_2(\xi))\end{matrix}\right).\]
Therefore, conditions \eqref{eq:condition_coercivity} imply that the above quadratic form is positive definite, and hence $\mF[\rho_1,\rho_2]> 0$ for all $\rho_1,\rho_2 \in L^2(\R)$. Now, assume by contradiction that there exists a
global minimiser $(\rho_1,\rho_2)$ with the prescribed mass constraints. Then $\mF[\rho_1,\rho_2]>0$. We now rescale $(\rho_1,\rho_2)$ by a parameter $\lambda>0$ in such a way to preserve the total mass of both components, i.e.
\[\rho_{1,\lambda}(x)=\lambda^{-1}\rho_1(\lambda^{-1}x),\qquad \rho_{2,\lambda}(x)=\lambda^{-1}\rho_2(\lambda^{-1}x),\]
notice that $(\rho_{1,\lambda},\rho_{2,\lambda})\in \mathcal{M}$. We can use \eqref{eq:functional_bounded} as follows
\[0\leq \mF[\rho_{1,\lambda},\rho_{2,\lambda}]
\leq C\left(\|\rho_{1,\lambda}\|_{L^2}^2 + \|\rho_{2,\lambda}\|_{L^2}^2\right)
=C\lambda^{-1}\left(\|\rho_{1}\|_{L^2}^2 + \|\rho_{2}\|_{L^2}^2\right) \ . \]
The latter right-hand side converges to zero as $\lambda \rightarrow +\infty$. Therefore, for a large enough $\lambda$ the value of the functional on $(\rho_{1,\lambda},\rho_{2,\lambda})$ can be made smaller than $\mF[\rho_1,\rho_2]$, thus contradicting the fact that $(\rho_1,\rho_2)$ is a nontrivial global minimiser.
\endproof

We now establish a reasonable necessary condition for the existence of a nontrivial steady state.
\begin{teo}
Under conditions \eqref{eq:condition_coercivity}, there exists no nonzero stationary solution for \eqref{eq.stationary}.
\end{teo}

\proof
Let $(\rho_1,\rho_2)\in \mathcal{M}$. We use Lemma \ref{lem:second_gateaux} to
compute the second derivative $H[\mu,\nu]$ of $\mF[\rho_1,\rho_2]$ in the direction $(\mu,\nu)$. If
\begin{equation}\label{eq:deter}
\int_{\R}(\ep\mu^{2}-\mu S_{1}*\mu)dx\int_{\R}(\ep\nu^{2}-\nu S_{2}*\nu)dx
> \left(\int_{\R}(\ep\mu\nu-\mu K*\nu)dx\right)^{2},
\end{equation}
then $\det(H[\mu,\nu])> 0$. Indeed, applying the Fourier transform on both sides of the inequality above, we obtain the equivalent condition
\begin{align*}
&\int_{\R}(\ep-\hat{S}_{1})\hat{\mu}^{2}d\xi\int_{\R}(\ep-\hat{S}_{2})\hat{\nu}^{2}d\xi > \left(\int_{\R}\hat{\mu}(\ep\hat{\nu}-\hat{K}\hat{\nu})d\xi\right)^{2}.
\end{align*}
Now from \eqref{eq:condition_coercivity}$(ii)$ we deduce
\begin{align*}
& \left(\int (\ep-\hat{K}(\xi))\hat{\mu}(\xi)\hat{\nu}(\xi)d\xi\right)^2 <\left(\int |\ep -\hat{S}_1(\xi)|^{1/2} |\ep -\hat{S}_2(\xi)|^{1/2} |\hat{\mu}(\xi)||\hat{\nu}(\xi)|d\xi\right)^2\\
& \quad \quad  \leq \left(\int |\ep-\hat{S}_1(\xi)| \hat{\mu}^2(\xi) d\xi\right) \left(\int |\ep-\hat{S}_2(\xi)| \hat{\nu}^2(\xi) d\xi\right),
\end{align*}
which implies $\det(H[\mu,\nu])> 0$. From
\eqref{eq:condition_coercivity}$(i)$ one deduces that the second derivative of $\mF$
in (arbitrary) direction $(\mu,\nu)$ is
positive definite. Hence, $\mF$ is
strictly convex on $L^2(\R)\times L^2(\R)$. Now, assume $(\bar{\rho}_1,\bar{\rho}_2)$ is a non trivial steady state. Since $\mF$ is strictly convex, then the critical point $(\bar{\rho}_1,\bar{\rho}_2)$ is also the unique global minimiser. This contradicts Lemma \ref{lem:coercivity}.
\endproof

\begin{oss}
\emph{Conditions \eqref{eq:condition_coercivity} generalise the diffusion-dominated
condition established in \cite{franek} for the one species case. In that case,
the only interaction kernel $S$ should satisfy $\|S\|_{L^1}>\ep$ in order to
ensure the existence of a non trivial steady state, or equivalently the condition $\ep\geq \|S\|_{L^1}$ implies that no global minima exist with fixed positive mass. In the case with two
species, the two conditions
$\|S_i\|_{L^1}<\ep $ $i=1,2$,
ensure that
\[\hat{S}_i(\xi)\leq \|\hat{S}_i\|_{L^\infty} \leq \|S_i\|_{L^1}<\ep\,,\qquad i=1,2,\]
and therefore the trace condition in \eqref{eq:condition_coercivity} is
trivially satisfied. This condition has a similar interpretation to the one
in \cite{franek} for the one species case, i. e. the diffusion part is
stronger than the (self) attraction parts, and so the spreading behaviour of
particles dominates in the large time dynamics. However,
\eqref{eq:condition_coercivity}$(ii)$ is a more specific feature of
the two species case. In order to provide a heuristic interpretation of such condition, let us consider for simplicity Gaussian potentials of the form
\[
S_{1}(x)=S_2(x)=S(x)=\frac{A}{\sigma\sqrt{\pi}}e^{- \frac{x^2}{2\sigma^2}}\,, \qquad K(x)=\frac{B}{\lambda\sqrt{\pi}}e^{- \frac{x^2}{2\lambda^2}} ,
\]
for some positive constants $A, B, \sigma, \lambda$. Applying the Fourier transform we see that condition (i) is equivalent to
\[\ep>A.\]
Condition
\eqref{eq:condition_coercivity}$(ii)$
is satisfied e.g. if
\begin{equation}\label{eq:coercivity2gaussian}
  A<B<\ep\,,\qquad \lambda\ll\sigma,
\end{equation}
i.e. if $K$ has a larger mass than $S$ and 
the variance of $K$ is much smaller compared to that of $S$. Roughly speaking, this means that cross interaction should be relevant only at very small distances between particles of different species. This is not surprising. Consider for instance two separated patterns for $\rho_1$ and $\rho_2$, and assume that condition \eqref{eq:condition_coercivity}$(i)$ is satisfied. Then, the lack of cross diffusion (which only appears when both $\rho_1$ and $\rho_2$ are positive) pushes the two patterns towards `spreading'. The only effect that could prevent the diffusive behaviour to dominate is the cross-interaction. Now, assume that \eqref{eq:condition_coercivity}$(ii)$ is satisfied, for instance in the form \eqref{eq:coercivity2gaussian}. Such a condition somehow ensures that the cross-interaction potential is not exerting any long-range confining effect on the particles to compensate the diffusive effect, because it only acts at small distances. When the two patterns touch each other, the cross-interaction potential will only interfere within the mixing area (still because of \eqref{eq:condition_coercivity}$(ii)$ ), still not enough to produce a `global confinement' and prevent the whole solution from decaying.}
\end{oss}

The above remark pinpoints an important fact about the minimization problem for $\mathcal{F}$: when the two species are separated and far from each other, the only coupling mechanism between them is the cross-interaction term ruled by the potential $K$. Hence, if we impose the conditions ensuring that the two species would feature a nontrivial global minimum in absence of any coupling ($\ep<A$ in the above example), we immediately see that the functional can be diminished by decreasing the distance between the two species. Hence, if the two species are separated at the global minimum state, common sense would suggest that their supports are glued together. Proving that this is actually what happens is the main goal of this paper, and is tackled in the next sections from different viewpoints.

\subsection{Structure of Energy Minimizers}

The previous interpretation of the supports of $\rho_1$ and $\rho_2$ staying separated (possibly gluing together at one edge of their boundaries) is supported by the next result, which refers to a specific case, namely $K$ being
sufficiently smooth with unique maximum at zero, and
$S_i = \sigma_i K$, $i=1,2$, with positive real numbers $\sigma_i$. Under
these conditions, provided that $\sigma_1+\sigma_2>2$, we can prove that the
supports of the two components must necessarily have an empty intersection at
local minimisers of $\mF$, which corresponds to segregation. The next theorem is true also in multiple dimensions, but
for simplicity we only state it in one space dimension.
\begin{teo}
Assume that $K\in C(\R)$  with a unique maximum at zero and being strictly radially decreasing in a neighbourhood. Let
$S_i=\sigma_i K$ for some $\sigma_1,\sigma_2>0$ with $\sigma_1+\sigma_2>2$.
Let $(\rho_1^\infty,\rho_2^\infty)$ be a local minimizer of
$\mF$ on $\mathcal{M}$, and
\begin{equation*}
\mathcal{S}:= \mbox{supp}(\rho_1^\infty)\cap \mbox{supp}(\rho_2^\infty).
\end{equation*}
Then $\mathcal{ S}$ has zero Lebesgue measure.

\end{teo}
\proof
Assume that $\mathcal{S}$ has positive Lebesgue measure, then we can choose a subset $\mathcal{O}$ of positive measure and some constant $C > 0$ such that
$$ \rho_i^\infty(x) \geq C, \qquad i=1,2, ~\text{for almost every } x \in \mathcal{O}.$$
For a fixed $d > 0$ (e.g. half the diameter of $\mathcal{O}$) we can choose $z_1, z_2 \in \mathcal{O}$ with distance greater or equal
than $d$ and $0 < \gamma < {d}/4$ sufficiently small such that the balls $B_\gamma(z_i)$ of radius $\gamma$ around $z_1$ and $z_2$ intersected with $\mathcal{O}$ have positive Lebesgue measure
and are disjoint. Define
$$ u_1(x) = \left\{ \begin{array}{ll} {\vert B_\gamma(z_1) \cap \mathcal{O} \vert}^{-1} & \mbox{if~} x \in B_\gamma(z_1) \cap \mathcal{O}\\
-{\vert B_\gamma(z_2) \cap \mathcal{O} \vert}^{-1} & \mbox{if~} x \in B_\gamma(z_2) \cap \mathcal{O} \\ 0 & \mbox{else} \end{array} \right. $$
and $u_2 = - u_1$.
Since $u_1, u_2 \in L^2(\R)$ have mean zero, for $\delta > 0$ sufficiently
small, $\rho_i^\infty \pm \delta u_i$ is nonnegative and has the same mass as $\rho_i^\infty$. Hence $(\rho_1^\infty \pm \delta u_1, \rho_2^\infty \pm \delta u_2)$ is
admissible for the minimization of $\mF$ with $\delta << 1$. It is hence straightforward to see that at a minimizer the first variation of $\mF$ in direction $(u_1,u_2)$ vanishes. With
the homogeneity of $\mF$,  we find
$$ \mF[\rho_1^\infty+\delta u_1,\rho_2^\infty+ \delta u_2] = \mF[\rho_1^\infty,\rho_2^\infty] + \delta^2 \mF[u_1,u_2]. $$
Now, due to $u_1+u_2=0$ we obtain
\begin{align*} \mF[u_1,u_2] &= - \int_{\R^d} \left( \frac{\sigma_1}2  u_1 K*u_1 + \frac{\sigma_2}2  u_2 K*u_2 + u_1 K*u_2  \right)~dx \\ & = - \frac{1}2 (\sigma_1 + \sigma_2 -2) \int_{\R^d}u_1 K*u_1~dx .
\end{align*}
And we have
\begin{align*}\int_{\R^d}u_1 K*u_1~dx =& {\vert B_\gamma(z_1) \cap \mathcal{O} \vert^{-2}}  \int_{B_\gamma(z_1) \cap \mathcal{O}}\int_{B_\gamma(z_1) \cap \mathcal{O}} K(x-y) ~dx~dy  \\
& + {\vert B_\gamma(z_2) \cap \mathcal{O} \vert^{-2}}  \int_{B_\gamma(z_2) \cap \mathcal{O}}\int_{B_\gamma(z_2) \cap \mathcal{O}} K(x-y) ~dx~dy \\
& - \frac{2}{\vert B_\gamma(z_1) \cap \mathcal{O} \vert~\vert B_\gamma(z_2) \cap \mathcal{O} \vert}  \int_{B_\gamma(z_1) \cap \mathcal{O}}\int_{B_\gamma(z_2) \cap \mathcal{O}} K(x-y) ~dx~dy
 \\ \geq & 2 \left( \inf_{|z|< 2\gamma} K(z) - \sup_{|z|>d-2\gamma} K(z) \right).
\end{align*}
For $\gamma$ sufficiently small, the local strict radial decrease around zero implies that the infimum is larger than the supremum above (note that $2 \gamma < d - 2\gamma$) and thus, $ \int_{\R^d}u_1 K*u_1~dx $ is positive.
Hence, $\mF[u_1,u_2] < 0$ and $(\rho_1^\infty,\rho_2^\infty)$ cannot be a minimizer of $\mF$.
%
\endproof

In the case of one weaker interaction, e.g. $\sigma_2 < 1$ one can give an improved result on the support by following the proof of \cite{novaga_et_al}, where
only the attractive part of the energy with an upper bound on the total
density $\rho_1+\rho_2$ was considered. The idea is to write the energy as a functional of $\rho_1+\rho_2$ and $\rho_2$ and apply Riesz-Sobolev symmetrization to these functions in order to obtain states of lower energy. Since such a symmetrization leaves the $L^2$-norm of $\rho_1+\rho_2$ unchanged it can be applied directly to our case:
\begin{teo}
Assume that $K\in C(\R)$  is strictly radially decreasing. Let
$\sigma_1,\sigma_2>0$ with $\sigma_1+\sigma_2>2$ and $\sigma_2 < 1$.
Let $(\rho_1^\infty,\rho_2^\infty)$ be a global minimiser with compact support of
$\mF$ on $\mathcal{M}$ with $m_1 = m_2$. Then $\rho_1^\infty + \rho_2^\infty$ is radially symmetric around its center of mass $X_0$ and there exist $b > a > 0$ such that
\begin{equation}
	\mbox{supp}(\rho_1^\infty) = [X_0-a,X_0+a], \quad \mbox{supp}(\rho_2^\infty) = [X_0-b,X_0-a] \cup [X_0+a,X_0+b].
\end{equation}
\end{teo}

Of course the phase separation does not mean that the supports are at positive distance, due to the attractive cross interactions we expect that the supports are glued together. This is indeed true in general for stationary solutions.  The proof is the same
as the corresponding result for a single species in \cite[Lemma 4.1]{franek}, we hence only provide a sketch of the proof.

\begin{prop}
Let $\sigma_1, \sigma_2 > 0$, $K$ be strictly decreasing with respect to the radial variable and $\left(\rho_1^\infty,\rho_2^\infty \right)\in \mathcal{M} \cap BV(\R)^2$ be a weak solution of \eqref{eq.stationary} according to Definition \ref{def:stationary}. Then $supp(\rho_1^\infty+\rho_2^\infty)$ is a connected interval.
\end{prop}

\begin{proof-sketch}
Let us just mention the main idea of the proof: Assume there exists an interval $[a,b]$ with $a < b$ not in the support of $\rho_1^\infty+\rho_2^\infty$, but such that the intersections of $(-\infty,a)$ and $(b,+\infty)$ with the support of $\rho_1^\infty+\rho_2^\infty$ are both nonempty. Then we can construct a smooth velocity $V$ field equal to $+1$ on $(-\infty,a)$ and $-1$ on $(b,\infty)$. Consider then the infinitesimal change of the energy on the curve $(u_1(s,x),u_2(s,x))$ solving (locally) the Cauchy problem
\begin{equation*}
  \begin{cases}
    \partial_s u_i + \partial_x (u_i V)=0 & \\
    u_i(0,x)=\rho_i^\infty(x) &
  \end{cases}
\end{equation*}
for $i=1,2$. Since $\frac{d}{ds} \mathcal{F}[u_1(s),\rho_2^\infty]|_{s=0} = \frac{d}{ds} \mathcal{F}[\rho_1^\infty,u_2(s)]|_{s=0} = 0$ in view of the fact that $\left(\rho_1^\infty,\rho_2^\infty \right)$ is a stationary state according to Definition \ref{def:stationary}, due to the computation
\begin{align*}
   & 0=\frac{\epsilon}{2}\int_{-\infty}^a\partial_x(\rho_1^\infty + \rho_2^\infty)^2 dx - \frac{\epsilon}{2}\int_{b}^{+\infty}\partial_x(\rho_1^\infty + \rho_2^\infty)^2 dx,
\end{align*}
with a few manipulations we obtain
\begin{align*}
   & \int_{-\infty}^a \rho_1^\infty S_1'\ast \rho_1^\infty dx + \int_{-\infty}^a \rho_1^\infty K'\ast \rho_2^\infty dx + \int_{-\infty}^a \rho_2^\infty S_2'\ast \rho_2^\infty dx \\
    & \ + \int_{-\infty}^a \rho_2^\infty K'\ast \rho_1^\infty dx = \int_{b}^{+\infty} \rho_1^\infty S_1'\ast \rho_1^\infty dx \\
    & \ +\int_{b}^{+\infty}  \rho_1^\infty K'\ast \rho_2^\infty dx + \int_{b}^{+\infty} \rho_2^\infty S_2'\ast \rho_2^\infty dx + \int_{b}^{+\infty} \rho_2^\infty K'\ast \rho_1^\infty dx.
\end{align*}
Now, since all the involved kernels $S_1, S_2, K$ are even, with a few manipulations similar to those in \cite[Lemma 4.1]{franek}, we get that the support of $\rho_1^\infty + \rho_2^\infty$ must be empty in at least one between $(-\infty,a)$ and $(b,+\infty)$, which is a contradiction. This proves that $(\rho_1^\infty,\rho_2^\infty)$ cannot be a stationary state according to Definition \ref{def:stationary}.
\end{proof-sketch}

We can provide the same result for an arbitrary minimizer of the interaction energy with a similar proof, respectively the Lagrangian version:
\begin{prop}
Let $\sigma_1, \sigma_2 > 0$, $K$ be strictly decreasing with respect to the radial variable and $\left(\rho_1^\infty,\rho_2^\infty \right)\in \mathcal{M}$ be a minimizer of $\mF$. Then $supp(\rho_1^\infty+\rho_2^\infty)$ is a connected interval.
\end{prop}
\begin{proof}
Suppose that the difference between $supp(\rho_1^\infty+\rho_2^\infty)$ and its convex hull contains an interval $[a,b]$. Let $0<d<\frac{b-a}2$ and define $(\tilde \rho_1,\tilde \rho_2) \in \mathcal{M}$ via
$$ \tilde \rho_i(x) = \left\{  \begin{array}{ll} \rho_i^\infty(x-d) & \text{if } x < a+d  \\
\rho_i^\infty(x+d) & \text{if } x >b-d \\
\rho_i^\infty(x) & \text{else }.
\end{array} \right. $$
Then we have
$$\int_\R (\tilde \rho_1+\tilde \rho_2)^2 ~dx = \int_\R (\rho_1^\infty+\rho_2^\infty)^2 ~dx, $$
while the attractive interaction terms in the energy are smaller for $(\tilde \rho_1,\tilde \rho_2)$. Hence, $\mF[\tilde \rho_1,\tilde \rho_2]   < \mF[\rho_1^\infty,\rho_2^\infty ] $, which is a contradiction.
\end{proof}

Despite the result of compact support for $\rho_1^\infty+\rho_2^\infty$, which
holds for arbitrary positive interactions, we expect mixing in case of dominant cross-attraction, i.e. $\sigma_1 <1$ and $\sigma_2 < 1$ similar to \cite{novaga_et_al}.

Note that Proposition 3.3 is not a special case of Proposition 3.2, since we cannot argue that $\rho_1^\infty+\rho_2^\infty$ are Lipschitz-continuous, hence we cannot apply the result Proposition 3.1. In order to do so we will need to verify the Lipschitz property, for which Proposition 3.3 appears to be crucial.
In this direction we now provide a result closing the gap between energy minimizers and
stationary solutions. At least for $\sigma_1+\sigma_2 > 2$ and $\sigma_2 < 1$ the assumptions of the following theorem are satisfied:
\begin{teo}
Assume $K\in C^1(\R)$  to be strictly radially decreasing. Let
$\sigma_1,\sigma_2>0$ and let $(\rho_1^\infty,\rho_2^\infty)$ be a global minimiser with compact support of
$\mF$ on $\mathcal{M}$.
Moreover, assume that for some $M > 1$ there exist $a_1 < a_2 < \ldots < a_M$
such that supp$(\rho_1^\infty + \rho_2^\infty) = \cup_{i=1}^{M-1} [a_i,a_{i+1}]$ and that for each $i=1,\ldots,M-1$ we have
\begin{equation}
	\rho_k^\infty > 0, \rho_\ell^\infty=0 \text{ on } (a_i,a_{i+1}),
\end{equation}
with $\{k,\ell\} = \{1,2\}$. Then $\rho_1^\infty + \rho_2^\infty$ is Lipschitz-continuous.
\end{teo}
\begin{proof}
Let $i$ be such that $k=1$ and $\ell=2$ (the opposite case is analogous). Moreover, let $\varphi \in C(\R)$ be an arbitrary continuous function supported in a compact subset $(a_i,a_{i+1})$. Then by standard arguments on the first variation we see that $\mF[\varphi,\rho_2^\infty]=0$, which implies
$$ \rho_1^\infty + \rho_2^\infty = \rho_1^\infty = \frac{1}\ep\int_R K(x-y) [\sigma_1 \rho_1^\infty(y) + \rho_2^\infty(y)] ~dy , $$
on $(a_i,a_{i+1})$. The continuous differentiability of $K$ implies that $\rho_1^\infty+\rho_2^\infty$ is Lipschitz continuous (with modulus independent of $i$) on $(a_i,a_{i+1})$. Hence, it suffices to show that $\rho_1^\infty+\rho_2^\infty$ is continuous on the finite set of points $\{a_i\}$ in order to obtain Lipschitz continuity on $\R$. Due to the uniform Lipschitz continuity on the subintervals, the left- and right-sided limits exist for each subinterval.

First, let $\rho_k^\infty > 0$ in $(a_1,a_2)$ and assume $\lim_{x\rightarrow a_1^+} \rho_k^\infty(x) > 0$. For $\delta_1, \delta_2 > 0$ being sufficiently small define
$$ \tilde \rho_k(x) = \left\{  \begin{array}{ll} \rho_k^\infty - \delta_1 & \text{for } x\in (a_1,a_1+\delta_2) \\ \rho_k^\infty + \delta_1 & \text{for } x\in (a_1,a_1-\delta_2) \\ \rho_k^\infty(x) &\text{else} \end{array} \right. \qquad \text{and} \qquad \tilde \rho_\ell = \rho_\ell^\infty.$$ Then it is straight-forward to show that
$$ \mF[\tilde \rho_1,\tilde \rho_2] = \mF[\rho_1^\infty,\rho_2^\infty] - \ep \delta_1 \delta_2 \lim_{x\rightarrow a_1^+} \rho_k^\infty(x) + \mathcal{O}(\delta_1^2 \delta_2 + \delta_1 \delta_2^2), $$
which strictly smaller than $\mF[\rho_1^\infty,\rho_2^\infty]$ for $\delta_1, \delta_2$ sufficiently small and hence a contradiction to $(\rho_1^\infty,\rho_2^\infty)$ being a minimiser. Thus,
$$ \lim_{x\rightarrow a_1^+} (\rho_1^\infty(x)+\rho_2^\infty(x)) = \lim_{x\rightarrow a_1^+} \rho_k^\infty(x) = 0. $$
A completely analogous proof shows continuity at $x=a_M$. Finally consider
the limit at $x=a_i$, $1<i<M$, if $ \lim_{x\rightarrow a_i^+} \rho_k^\infty(x) = \lim_{x\rightarrow a_i^+} \rho_\ell^\infty(x)$. Define a similar perturbation
in a neighbourhood $(a_i-\delta_2,a_i+\delta_2)$ increasing the smaller and
decreasing the larger density. With the same kind of expansion as at $a_1$ one
can obtain a state with lower energy, which is a contradiction.
\end{proof}

\section{Existence of segregated states via Krein-Rutman theorem}\label{sec:krein}

Hinted by the results in the previous section, we now provide reasonable
sufficient conditions for the existence of segregated stationary states for
system \eqref{eq.stationary}. We remark that the system \eqref{eq.stationary} is translation invariant. Therefore, we shall fix the center of mass to zero for simplicity.

\begin{defin}\label{sgr}
We call a stationary solution $(\rho_1,\rho_2)$ to \eqref{eq.stationary} according to Definition \ref{def:stationary} a \emph{symmetric segregated
steady state} if there exist $L_1, L_2>0$ such that
\begin{equation*}
\mbox{supp}(\rho_1)=[-L_{1},L_{1}]=:I_{1} \quad \mbox{and} \quad
\mbox{supp}(\rho_2)=[-L_2,-L_1]\cup [L_{1},L_{2}]=:I_{2} \, ,
\end{equation*}
where $\rho_1, \rho_2$ are even functions, and $C^1$-regular inside their
respective supports, such that $w=\rho_1+\rho_2$ is monotone decreasing on
$[0,L_2]$ with $w(L_2)=0$.
\end{defin}

Recall that Definition \ref{def:stationary} prescribes
$w(x):= \rho_1(x)+\rho_2(x)$ being Lipschitz continuous on $\R$.


Our strategy to prove existence of segregated steady states is a non-trivial extension of the strategy proposed in \cite{franek}, and adapted in
\cite{burger_fetecau_huang}. The main idea is to fix $L_1$ and $L_2$ and then
look at the stationary equations for $\rho_1$ and $\rho_2$ as eigenvalue
conditions for a suitable integral operator. The existence of the eigenvectors
will be proven by using the strong version of the Krein-Rutman theorem.

\begin{teo}[Krein-Rutman]
\label{K-R}
Let $X$ be a Banach space, $K\subset X$ be a solid cone, such that
$\lambda K\subset K$ for all $\lambda\geq0$ and $K$ has a nonempty interior
$K^{o}$. Let $T$ be a compact linear operator on X, which is strongly positive with respect to $K$, i.e. $T[u]\in K^{o}$ if $u\in K\setminus\{0\}$. Then,
\begin{itemize}
\item [(i)] the spectral radius $r(T)$ is strictly positive and $r(T)$ is a simple eigenvalue with an eigenvector $v\in K^{o}$. There is no other eigenvalue with a corresponding eigenvector $v\in K$.
\item [(ii)] $|\lambda|<r(T)$ for all other eigenvalues $\lambda\neq r(T)$.
\end{itemize}
\end{teo}

\noindent We shall work with the following class of kernels:
\begin{equation}\label{eq:kernels_main_assumption}
\hbox{$S_1, S_2, K$ are symmetric and strictly decreasing on $[0,+\infty)$},
\end{equation}
which, together with our nonnegativity assumption (A3) imply in particular
that all the kernels are supported on $\R$. \\
Moreover, having fixed $L_1<L_2$, we assume
for all $x \in (0,L_1)$ that
\begin{equation}\label{eq:assumption_K1}
S_{1}(x-L_{1})-S_{1}(x+L_{1}) > K(x-L_{1})-K(x+L_{1}),
\end{equation}
\begin{equation}\label{eq:assumption_K2}
\mbox{and } \,  S_{2}(x-L_{1})-S_{2}(x+L_{1}) < K(x-L_{1})-K(x+L_{1}),
\end{equation}
for all $x \in (L_1,L_2)$. Assumptions \eqref{eq:assumption_K1} and \eqref{eq:assumption_K2} are met for instance in the significant case $S_i = \sigma_i K$, $\sigma_1 > 1 > \sigma_2$. We assume further that
\begin{equation}\label{eq:assumption_K3}
S_1'(L_1) < K'(L_1).
\end{equation}
Let $\rho_1$ and $\rho_2$ be symmetric steady states, then we can rephrase
\eqref{eq.stationary} as
\begin{equation}
\label{eq.constant}
\begin{cases}
		\ep(\rho_1+\rho_2)-S_{1}\ast\rho_1-K\ast\rho_2=C_{1} & \hbox{for $x\in I_1$}\\
	      \ep(\rho_1+\rho_2)-S_{2}\ast\rho_2-K\ast\rho_1=C_{2} & \hbox{for $x\in I_2$},
\end{cases}
\end{equation}
for suitable constants $C_1, C_2>0$. Assuming segregation, we rewrite \eqref{eq.constant} as
\begin{equation}\label{eq:system_rho}
\begin{cases}
		\displaystyle{x\in I_{1}: \quad \ep \rho_1(x)
=\int_{\R}S_{1}(x-y)\rho_1(y)dy+\int_{\R}K(x-y)\rho_2(y)dy+C_{1}} \\
\\
	       \displaystyle{x\in I_{2}: \quad \ep \rho_2(x)
=\int_{\R}S_{2}(x-y)\rho_2(y)dy+\int_{\R}K(x-y)\rho_1(y)dy+C_{2}}
\end{cases}.
\end{equation}
Let $\bar{w}=\rho_1(L_{1})=\rho_2(L_{1})$. Let $p(x)=-\rho_1'(x)$ be
restricted to the interior of $I_1$ and $q(x)=-\rho_2'(x)$ to the interior of
$I_2$. By differentiating \eqref{eq:system_rho} we obtain
\begin{equation}\label{eq:deriv}
\begin{cases}
      \displaystyle{x\in I_{1}: \, \ep p(x)
=\int_{I_{1}}S_{1}(x-y)p(y)dy+\int_{I_{2}}K(x-y)q(y)dy +\bar{w}A_1(x)} \\
\\	
          \displaystyle{ x\in I_{2}: \,
\ep q(x)=\int_{I_{2}}S_{2}(x-y)q(y)dy+\int_{I_{1}}K(x-y)p(y)dy +
\bar{w} A_2(x)}
\end{cases},
\end{equation}
where
\begin{align}
A_1(x) &= S_{1}(x-L_{1})-S_{1}(x+L_{1})+K(x+L_{1})-K(x-L_{1}), \\
A_2(x) &= K(x-L_{1})-K(x+L_{1})+S_{2}(x+L_{1})-S_{2}(x-L_{1}).
\end{align}
A symmetrization yields for $x>0$, that
\begin{eqnarray*}
\ep p(x)&=&\int_{0}^{L_{1}} \big( S_{1}(x-y)-S_{1}(x+y)\big)p(y)dy\\
&& + \int_{L_{1}}^{L_{2}}\big (K(x-y)-K(x+y)\big)q(y)dy+\bar{w}A_1(x),\\
\ep q(x)&=&\int_{L_{1}}^{L_{2}}\big (S_{2}(x-y)-S_{2}(x+y)\big)q(y)dy\\
&&+\int_{0}^{L_{1}}\big(K(x-y)-K(x+y)\big)p(y)dy+\bar{w}A_2(x).
\end{eqnarray*}
To simplify notation, define the non-negative functions, for $x,y>0$,
\begin{equation*}
\bar{G}(x,y)=G(x-y)-G(x+y)\, , \mbox{ for } \,  G=S_1, S_2, K \ .
\end{equation*}
Then we can rewrite system \eqref{eq:deriv} as
\begin{align*}
&\ep p(x)=\int_{0}^{L_{1}}\bar{S}_{1}(x,y)p(y)dy+\int_{L_{1}}^{L_{2}}\bar{K}(x,y)q(y)dy+\bar{w} A_1(x),\\
&\ep q(x)=\int_{L_{1}}^{L_{2}}\bar{S}_{2}(x,y)q(y)dy+\int_{0}^{L_{1}}\bar{K}(x,y)p(y)dy+\bar{w} A_2(x).
\end{align*}

The result of this section is stated in the following Proposition.
\begin{prop}
Let assumptions (A1), (A3), \eqref{eq:kernels_main_assumption},
\eqref{eq:assumption_K1}, \eqref{eq:assumption_K2}, and
\eqref{eq:assumption_K3} be satisfied. Let $0<L_1<L_2$ be fixed. Then, there
exists a unique (up to mass normalisation) symmetric segregated steady state
$\left(\rho_1,\rho_2\right)$ to \eqref{eq.stationary} with
$\ep=\ep(L_1,L_2)>0$.
\end{prop}
\proof
Consider the Banach space
\begin{equation*}
X_{L_1,L_{2}}=\left\{P=(p,q;w)\in C^{1}([0,L_{1}])\times C^1([L_1,L_2])\times \R,\, \mid p(0)=0\right\},
\end{equation*}
equipped with the $W^{1,\infty}$-norm for the first two components and with the standard one-dimensional Euclidean norm for the third component.
We use the notation $P=(p,q;w)$ for all $P\in X_{L_1,L_{2}}$, with $p\in C^{1}([0,L_{1}])$, $q\in C^1([L_1,L_2])$, and $w\geq 0$. For a given $P\in X_{L_1,L_{2}}$ define $T_{L_1,L_{2}}[P]\in X_{L_1,L_{2}}$ as
\[T_{L_1,L_{2}}[P]=(f,g;w') \in C^{1}([0,L_{1}])\times C^1([L_1,L_2])\times
\R \, , \mbox{ with} \]
\begin{align*}
& f(x)=\int_{0}^{L_{1}}\bar{S}_{1}(x,y )p(y)dy+\int_{L_{1}}^{L_{2}}\bar{K}(x,y)q(y)dy+w A_1(x),\qquad x\in [0,L_1]\\
& g(x)=\int_{L_{1}}^{L_{2}}\bar{S}_{2}(x,y)q(y)dy+\int_{0}^{L_{1}}\bar{K}(x,y)p(y)dy+w A_2(x), \qquad x\in [L_1,L_2]\\
& w' = \int_{L_1}^{L_{2}}\left(\int_{L_{1}}^{L_{2}}\bar{S}_{2}(x,y)q(y)dy+\int_{0}^{L_{1}}\bar{K}(x,y)p(y)dy+w A_2(x)\right)dx.
\end{align*}
The operator $T_{L_1,L_{2}}$ is compact on the Banach space $X_{L_1,L_{2}}$ as all the involved kernels are $C^2$ on compact intervals, and hence by Arzel\'{a}'s theorem the image of the unit ball in $X_{L_1,L_{2}}$ is pre-compact.
It is easy to show that the set	
\begin{equation*}
K_{L_1,L_{2}}=\left\{P=(p,q,w)\in X_{L_1,L_{2}} \mid p\geq 0, q\geq 0, w \geq 0\right\},
\end{equation*}		
is a solid cone in $X_{L_1,L_{2}}$ and that
\begin{eqnarray*}
H_{L_{1},L_{2}}&=& \left\{P=(p,q,w)\in K_{L_1,L_{2}} \mid p'(0)>0,\, p(x)>0\,\, \forall x\in ]0,L_{1}],\,\right. \\
&& \quad  \left. q(x)>0\,\, \forall x\in [L_{1},L_2],\,\, w>0 \right\}\subset K_{L_1,L_{2}}^{o},
\end{eqnarray*}
where $K_{L_1,L_{2}}^{o} $ denotes the interior of $K_{L_1,L_{2}}$.
We show that $T$ is strongly positive on the solid cone. First observe that,
due to \eqref{eq:assumption_K1} and \eqref{eq:assumption_K2}, all components
of $T_{L_1,L_{2}}[(p,q;w)](x)$ are nonnegative if $(p,q;w)\in K_{L_1,L_2}$ with $(p,q;w)\neq 0$. Moreover, setting $f(x)$ as the first component of $T_{L_1,L_{2}}[(p,q;w)](x)$ as above, we get
\begin{eqnarray*}
&&   \frac{d}{dx}f(x)\mid_{x=0}= \int_{0}^{L_{1}}\bar{S}_{1,x}(0,y)p(y)dy+\int_{L_{1}}^{L_{2}}\bar{K}_x(0,y)q(y)dy+ w A_1'(0)\\
&& \quad =\int_{0}^{L_{1}}\big ( S'_{1}(-y)-S'_{1}(y)\big)p(y)dy
+\int_{L_{1}}^{L_{2}}\big(K'(-y)-K'(y)\big)q(y)dy\\
&& \quad \quad  + w \Big(S'_1(-L_1)-S'_1(L_1)+K'(L_1)-K'(-L_1)\Big)\\
&&\quad =-2\int_{0}^{L_{1}}S'_{1}(y)p(y)dy-2\int_{L_{1}}^{L_{2}}K'(y)q(y)dy
\, -2w S_1'(L_1) + 2w K'(L_1) >0,
\end{eqnarray*}
since $(p,q;w)\neq 0$. Hence, $T_{L_1,L_2}[(p,q,w)]$ belongs to $H_{L_1,L_2}$ and therefore the operator is strongly positive on the cone $K_{L_1,L_2}$. The strong version of the Krein-Rutman theorem (Theorem \ref{K-R}) then gives the existence of an eigenvalue $\ep=\ep(L_1,L_2)$ such that
\[T_{L_1,L_2}[(p,q,\bar{w})]=\ep\cdot (p,q,w),\]
with eigenspace generated by a vector $(q,p,w)\in K_{L_1,L_2}^o$. Now, let $\rho_2(x)=\int_x^{L_2} q(x) dx$. The second and third equation
in the eigenvalue condition imply $\bar{w}=\rho_2(L_1)$. Let
$\rho_1(x)=\bar{w}+\int_x^{L_1}p(x) dx$. Clearly, $\rho_1(L_1)=\rho_2(L_1)$. Then, a computation similar to the one that led to the definition of the operator $T_{L_1,L_2}$ after differentiating \eqref{eq:system_rho} implies that $\rho_1$ and $\rho_2$ are symmetric segregated steady states to \eqref{eq.stationary}.
\endproof

\begin{oss}
\emph{The above result shows that a segregated steady state $(\rho_1,\rho_2)$ exists for arbitrary positive constants $L_1$ and $L_2$ such that $\mathrm{supp}(\rho_1)=[-L_1,L_1]$ and $\mathrm{supp}(\rho_2)=[-L_2,-L_1]\cup [L_1,L_2]$ for some diffusion constant $\epsilon>0$ which depends on $L_1$ and $L_2$. Clearly, we obtain a one-parameter family of solutions upon multiplication of $(\rho_1,\rho_2)$ by an arbitrary positive constant. Such a constant is uniquely determined by the total mass $m=m_1+m_2$ with $m_1=\int_\R \rho_1(x) dx$ and $m_2=\int_\R \rho_2(x) dx$. The two values $m_1$ and $m_2$ are not determined explicitly here. We expect the value of $L_1$ to be determined uniquely once $L_2$ and the two masses $m_1$ and $m_2$ are fixed. Most importantly, this approach does not provide an explicit information on the value of the diffusion constant $\epsilon$. A similar situation occurs in the problem studied in \cite{burger_fetecau_huang} for the one species equation with general power-law diffusion. Reasoning as in \cite{franek}, we expect that the diffusion constant $\epsilon$ obeys a monotone increasing law of the form $\epsilon=\overline{\epsilon}(L_2)$. If such a conjecture is true, then for any given $\epsilon$ in the range of the map $\overline{\epsilon}$ and for any given $m_1,m_2>0$ there exist unique $L_1,L_2>0$ and a unique (up to $x$-translations) segregated state $(\rho_1,\rho_2)$ with $m_1=\int_\R \rho_1(x) dx$, $m_2=\int_\R \rho_2(x) dx$, $\mathrm{supp}(\rho_1)=[-L_1,L_1]$, and $\mathrm{supp}(\rho_2)=[-L_2,-L_1]\cup [L_1,L_2]$. Such a conjecture will be addressed in full generality in a future study. In the next section we are able
to prove that \emph{such a statement is true for a small enough diffusion
coefficient}.}
\end{oss}

\section{Existence and uniqueness of steady states for small diffusion}\label{sec:small_diffusion}
Here we prove existence and uniqueness of a symmetric segregated steady state
for fixed masses, and a fixed \emph{small} diffusion coefficient.
We refer to Definition
\ref{sgr} to denote segregated states. Similar to \cite{budif}, we formulate the problem in the pseudo-inverse formalism and then use an implicit function theorem argument. Throughout this section we shall require
\begin{equation}\label{eq:c1c2}
  \min\{S_1''(0),\,S_2''(0),\,K''(0)\}>0.
\end{equation}
The main result of this section reads as follow.
\begin{teo}[Existence of segregated steady states for small diffusion]\label{teo_main}
Assume that the interaction kernels fulfil the additional assumption \eqref{eq:c1c2}. Then, there exists a constant $\epsilon_0$ such that for all $\epsilon\in (0,\epsilon_0)$ the stationary equation \eqref{eq.stationary} admits a unique solution in
the sense of Definition \ref{sgr} with fixed masses $m_1=\int \rho_1 dx$ and $m_2=\int \rho_2 dx$.
\end{teo}

\subsection{Formulation in the pseudo-inverse variable}\label{subsec:formulation}
Assume that $(\rho_1,\rho_2)$ is a segregated state with the structure as in
Definition \ref{sgr}, and
\[\int \rho_1(x) dx = z_1\,,\qquad \int \rho_2(x) dx = 1-z_1.\]
Hence, $w=\rho_1+\rho_2$ has unit mass and is supported on $[-L_2,L_2]$.
Let
\[F(x)=\int_{-\infty}^x w(y) dy.\]
Let $u:[0,1)\rightarrow \R$ be the pseudo-inverse of $F$
\[u(z)=\inf\{x\in \R\,:\,\, F(x)\geq z\}.\]
Set
$u_i(z)=u(z) \mathbf{1}_{J_i}(z)$, $i=1,2$, with
\[
\mbox{supp}(u_1)=\left[\frac{1}{2}-\frac{z_{1}}{2},\frac{1}{2}+\frac{z_{1}}{2}\right]=:J_{1},
\]
\[
\mbox{supp}(u_2)=\left[0,\frac{1}{2}-\frac{z_{1}}{2}\right]\cup \left[\frac{1}{2}+\frac{z_{1}}{2},1\right]=:J_{2}.
\]
Then \eqref{eq.stationary} can be rewritten as
\begin{equation*}
\begin{cases}
\ep\partial_x w(u_1(z))=\int_{J_1} S'_{1}\big(u_1(z)-u_1(\zeta)\big)d\zeta+\int_{J_2} K'\big(u_1(z)-u_2(\zeta)\big)d\zeta\,,\quad z\in J_1\,, \\
\ep\partial_x w(u_2(z))=\int_{J_2} S'_{2}\big(u_2(z)-u_2(\zeta)\big)d\zeta+\int_{J_1} K'\big(u_2(z)-u_1(\zeta)\big)d\zeta\,,\quad z\in J_2\,.
\end{cases}
\end{equation*}
Recall that $\rho_1$ and $\rho_2$ are symmetric, which implies that the
pseudo-inverse $u$ satisfies $u(1-z)=-u(z)$. Moreover, since $w$ is strictly
positive on $(-L_2,L_2)$, $u$ is strictly increasing on $(0,1)$ and Lipschitz continuous
on the compact subintervals of $(0,1)$. Due to $w$ being zero at $z=0$ and $z=1$, $u$ is expected to have an infinite slope at the boundary.

Since $w=\rho_1 + \rho_2$ and since $\rho_1$, $\rho_2$ have disjoint
supports, we have
\begin{equation}
\label{pseudo_stat}
\begin{cases}
\displaystyle{\frac{\ep}{2}\partial_z (\partial_z u_1(z))^{-2}
=\int_{J_1} S'_{1}\big(u_1(z)-u_1(\zeta)\big)d\zeta} \\
\qquad \qquad \qquad \qquad \quad \displaystyle{+ \int_{J_2}K'
\big(u_1(z)-u_2(\zeta)\big)d\zeta}, & z\in J_1, \\
\displaystyle{\frac{\ep}{2}\partial_z(\partial_z u_2(z))^{-2}
=\int_{J_2} S'_{2}\big(u_2(z)-u_2(\zeta)\big)d\zeta} \\
\qquad \qquad \qquad \qquad \quad \displaystyle{+ \int_{J_1}K'
\big(u_2(z)-u_1(\zeta)\big)d\zeta}, & z\in J_2.
\end{cases}
\end{equation}
The idea is to solve \eqref{pseudo_stat} for small $\ep$, hinted by the
fact that the case $\ep=0$ has the unique solution $u_1\equiv u_2\equiv 0$, which corresponds to $\rho_1$ and $\rho_2$ being two Dirac's deltas with masses $z_1$ and $1-z_1$ respectively. Similar to \cite{BuDiF_NHM}, we expect that the support of $w$ is small for small $\ep$. This suggests the
linearisation $u_i=\delta v_i$, $i=1,2$, with $v_1$, $v_2$ being odd
functions defined on $J_1$, $J_2$ respectively. A simple scaling
argument suggests that $\delta=\ep^{\frac{1}{3}}$, and therefore
\[
\begin{cases}
z\in J_1\quad\frac{\delta}{2}\partial_z(\partial_z v_1)^{-2}=
\int_{J_1} S'_{1}\big(\delta(v_1-v_1(\zeta))\big)+\int_{J_2} K'
\big(\delta(v_1-v_2(\zeta))\big), \\
z\in J_2\quad\frac{\delta}{2}\partial_z(\partial_z v_2)^{-2}
=\int_{J_2} S'_{2}\big(\delta(v_2-v_2(\zeta))\big)+\int_{J_1} K'
\big(\delta(v_2-v_1(\zeta))\big).
\end{cases}
\]
Multiplying the first equation by $\delta\partial_z v_1$ and the second one
by $\delta\partial_z v_2$, we get
\begin{align*}
&\delta^2\partial_z(\partial_z v_1)^{-1}=\partial_z\int_{J_1} S_{1}
\big(\delta(v_1(z)-v_1(\zeta))\big)+\partial_z\int_{J_2} K
\big(\delta(v_1(z)-v_2(\zeta))\big),\\
& \delta^2\partial_z(\partial_z v_2)^{-1}=\partial_z\int_{J_2} S_{2}
\big(\delta(v_2(z)-v_2(\zeta))\big)+\partial_z\int_{J_1} K
\big(\delta(v_2(z)-v_1(\zeta))\big).
\end{align*}
Taking the primitives w.r.t. $z$, we obtain for
$z\in J_1$, respectively $z \in J_2$  that
\begin{equation}\label{gorilla1}
\frac{\delta^2}{\partial_z v_1}=\int_{J_1} S_{1}\big(\delta(v_1(z)-v_1(\zeta))
\big)d\zeta+\int_{J_2} K\big(\delta(v_1(z)-v_2(\zeta))\big)d\zeta+\alpha_1,
\end{equation}
\begin{equation}\label{gorilla2}
\frac{\delta^2}{\partial_z v_2}=\int_{J_2} S_{2}\big(\delta(v_2(z)-v_2(\zeta))
\big)d\zeta+\int_{J_1} K\big(\delta(v_2(z)-v_1(\zeta))\big)d\zeta+\alpha_2,
\end{equation}
with integration constants $\alpha_1$, $\alpha_2$,
which are obtained by
substituting $z=1$ into (\ref{gorilla2})
\begin{equation}\label{eq:alpha2}
\alpha_2=-\int_{J_2} S_{2}\big(\delta(v_2(1)-v_2(\zeta))\big)d\zeta
-\int_{J_1} K\big(\delta(v_2(1)-v_1(\zeta))\big)d\zeta,
\end{equation}
and imposing the continuity condition for $(\partial_z u)^{-1}$ in
$\tilde{z}=(1+z_1)/2$
\begin{eqnarray*}
\alpha_1&=&\int_{J_2} S_{2}\big(\delta\big(v_2((1+z_1)/2)-v_2(\zeta)\big)\big)
-S_{2}\big(\delta(v_2(1)-v_2(\zeta))\big) \, d\zeta\nonumber\\
&& +\int_{J_1} K\big(\delta\big(v_2((1+z_1)/2)-v_1(\zeta)\big)\big)-
K\big(\delta(v_2(1)-v_1(\zeta))\big) \, d\zeta\nonumber\\
&& -\int_{J_1} S_{1}\big(\delta\big(v_1((1+z_1)/2)-v_1(\zeta)\big)\big)d\zeta\nonumber\\
&&
-\int_{J_2} K\big(\delta\big(v_1((1+z_1)/2)-v_2(\zeta)\big)\big)d\zeta.
\end{eqnarray*}
The symmetry requirements imply
\[\int_{J_1} v_1(\zeta)d\zeta=\int_{J_2} v_2(\zeta)d\zeta = 0.\]
Now we multiply \eqref{gorilla1} by $\delta \partial_z v_1$,
\eqref{gorilla2}
by $\delta \partial_z v_2$, integrate w.r.t. $z$ and introduce $G_i, H$ such
that $G_i'=S_i$, $i=1,2$, and $H'=K$. Here $G_1, G_2, H$ can be chosen odd
with $G_1(0)=G_2(0)=H(0)$.
The integration constants can be recovered by prescribing
\begin{equation*}
v_1(1/2)=0\, , \,
v_2(1)=\lambda\, , \,
v_1(\tilde{z})=v_2(\tilde{z})=\mu\, .
\end{equation*}
After some manipulations we obtain
\begin{eqnarray}
z-\frac{1}{2}&=&\delta^{-3}\left[\int_{J_1} G_1\big(\delta(v_1(z)
-v_1(\zeta))\big)
-\delta v_1(z) S_1\big(\delta(\mu-v_1(\zeta))\big) \, d\zeta\right.\nonumber\\
&&  \quad \quad \left. +\int_{J_2} H\big(\delta(v_1(z)-v_2(\zeta))\big)
-\delta v_1(z) K\big(\delta(\mu-v_2(\zeta))\big) \, d\zeta\right]\nonumber \\
&&  + \delta^{-2} v_1(z)\left[\int_{J_2} S_2\big(\delta(\mu-v_2(\zeta))\big)
-S_2\big(\delta(\lambda-v_2(\zeta))\big) \, d\zeta\right. \label{gorilla11}\\
&&  \quad \quad \quad \quad \quad \left. + \int_{J_1} K\big(\delta(\mu-v_1(\zeta))\big)
-K\big(\delta(\lambda-v_1(\zeta))\big) \, d\zeta\right]\, , \, z\in [1/2,\tilde{z}]\,,\nonumber
\end{eqnarray}
as well as
\begin{eqnarray}
z-\tilde{z} &=& \delta^{-3}\left[\int_{J_2} G_2\big(\delta(v_2(z)-v_2(\zeta))\big)
-\delta v_2(z) S_2\big(\delta(\lambda-v_2(\zeta))\big) \, d\zeta\right.\nonumber\\
&& \quad \quad \left.+\int_{J_1} H\big(\delta(v_2(z)-v_1(\zeta))\big)
-\delta v_2(z) K\big(\delta(\lambda-v_1(\zeta))\big) \, d\zeta\right]\nonumber \\
&& - \delta^{-3}\left[\int_{J_2} G_2\big(\delta(\mu-v_2(\zeta))\big)
-\delta \mu S_2\big(\delta(\lambda-v_2(\zeta))\big)\,
d\zeta\right.\label{gorilla22} \\
&& \quad \quad \quad  \left.+\int_{J_1} H\big(\delta(\mu-v_1(\zeta))\big)
-\delta \mu K\big(\delta(\lambda-v_1(\zeta))\big) \, d\zeta\right]\,,\quad z\in [\tilde{z},1]\,.\nonumber
\end{eqnarray}
A solution $(v_1,v_2)$ to \eqref{gorilla11}, \eqref{gorilla22} should be extended to $[0,1/2]$ to obtain an odd profile on the whole interval $[0,1]$.

\begin{oss}\label{rem:referee1}
  \emph{The functional system \eqref{gorilla11}-\eqref{gorilla22} has been obtained manipulating the original system \eqref{gorilla1}-\eqref{gorilla2} through integrations w.r.t. the independent variable $z$ and dividing by $\delta^2$. Therefore, seeking for a solution to \eqref{gorilla11}-\eqref{gorilla22} for $\delta=0$ somehow involves the computation of the \emph{first order term} in the expansion of the r.h.s. of \eqref{gorilla1}-\eqref{gorilla2} with respect to $\delta^2$. This computation will be performed in the next subsection. In some sense, in terms of the linearisation $u_i=\delta v_i$, we are computing the first order term w.r.t. $\delta$ of $u_i$ near $\delta=0$.}
\end{oss}

\subsection{A functional equation}\label{subsec:functional}

System \eqref{gorilla11}, \eqref{gorilla22} can be viewed as a functional
equation in the following form. Introduce the Banach space
\begin{align*}
\Omega=& \Big\{(v_1,\mu,v_2,\lambda)\in L^\infty([1/2,\tilde{z}])\times
\R\times L^\infty([\tilde{z},1])\times \R\,:\\
& \  \hbox{ $v_1$ is right continuous at $1/2$ and left continuous at
$\tilde{z}$}\,,\\
& \  \hbox{ $v_2$ is right continuous at $\tilde{z}$ and left continuous
at $1$,}\\
& \  \, \, v_1(1/2)=0\,,\quad v_1(\tilde{z})=v_2(\tilde{z})=\mu\,,\quad
v_2(1)=\lambda \Big\}.
\end{align*}
And consider the standard norm on $\Omega$
\[\interleave (v_1,\mu,v_2,\lambda)\interleave =\|v_1\|_{L^\infty} + \|v_2\|_{L^\infty} + |\mu| + |\lambda|.\]
For $\alpha>0$, consider the norm
\[\interleave (v_1,\mu,v_2,\lambda)\interleave_\alpha:=\interleave (v_1,\mu,v_2,\lambda)\interleave+\sup_{z\in [\tilde{z},1]}\frac{|\lambda-v_2(z)|}{(1-z)^\alpha},\]
and set
$\Omega_\alpha:=\left\{(v_1,\mu,v_2,\lambda)\in \Omega\,:\,\, \interleave (v_1,\mu,v_2,\lambda)\interleave_\alpha<+\infty\right\}$. \\
We now define the convex subset $U_\alpha\subset \Omega$
\[U_\alpha:=\left\{(v_1,\mu,v_2,\lambda)\in \Omega_\alpha\,:\,\, \hbox{$v_1$ and $v_2$ are increasing}\right\}.\]

In order to formulate \eqref{gorilla11}, \eqref{gorilla22} only on $[1/2,1]$,
we again use the \emph{symmetrised} function
\[\bar{G}(x;y):=G(x-y)+G(x+y)\,,\quad x,y>0\,,\]
for given $G:\R\rightarrow \R$. Moreover, we use the notation
$\tilde{J}_i= J_i\cap [1/2,1]\,,\quad i=1,2$.
Let $(v_1,\mu,v_2,\lambda)\in \Omega$, and $\delta>0$. Define
\begin{eqnarray}
&& F_1[(v_1,\mu,v_2,\lambda);\delta](z):=\frac{1}{2}-z\nonumber\\
&& \quad\quad  +\delta^{-3}\left[\int_{\tilde{J}_1} \bar{G}_1(\delta v_1(z);\delta v_1(\zeta))-\delta v_1(z) \bar{S}_1(\delta \mu;\delta v_1(\zeta)) \, d\zeta\right.\nonumber\\
&& \quad\quad\quad \quad \quad \left.+
\int_{\tilde{J}_2} \bar{H}(\delta v_1(z);\delta v_2(\zeta))-
\delta v_1(z) \bar{K}(\delta \mu;\delta v_2(\zeta)) \, d\zeta\right]\nonumber \\
&& \quad \quad  + \delta^{-2} v_1(z)\left[\int_{\tilde{J}_2}
\bar{S}_2(\delta \mu;\delta v_2(\zeta))-\bar{S}_2(\delta \lambda;\delta v_2(\zeta)) \, d\zeta\right.\nonumber\\
&& \quad\quad\quad \quad \quad \quad \quad  \left. + \int_{\tilde{J}_1}
\bar{K}(\delta \mu;\delta v_1(\zeta))-\bar{K}(\delta \lambda;\delta v_1(\zeta))\, d\zeta\right]\,,\quad z\in \tilde{J}_1\,,\label{Fgorilla11}
\end{eqnarray}
\begin{eqnarray}
&& F_2[(v_1,\mu,v_2,\lambda);\delta](z):=\tilde{z} -z\nonumber \\
&& \quad \quad + \delta^{-3}\left[\int_{\tilde{J}_2}
 \bar{G}_2(\delta v_2(z);\delta v_2(\zeta))-\delta v_2(z)
\bar{S}_2(\delta \lambda;\delta v_2(\zeta)) \, d\zeta\right.\nonumber\\
&& \quad \quad \quad \quad \quad \left.+\int_{\tilde{J}_1}
 \bar{H}(\delta v_2(z);\delta v_1(\zeta))-\delta v_2(z)
\bar{K}(\delta \lambda;\delta v_1(\zeta)) \, d\zeta\right]\nonumber \\
&& \quad \quad  - \delta^{-3}\left[\int_{\tilde{J}_2}
\bar{G}_2(\delta \mu;\delta v_2(\zeta))-\delta \mu \bar{S}_2
(\delta \lambda;\delta v_2(\zeta)) \, d\zeta\right.\nonumber\\
&& \quad \quad \quad \quad \quad  \left.+\int_{\tilde{J}_1} \bar{H}(\delta \mu;\delta v_1(\zeta))-\delta \mu \bar{K}(\delta \lambda;\delta v_1(\zeta)
)\, d\zeta\right]\,,\quad z\in \tilde{J}_2\,.\label{Fgorilla22}
\end{eqnarray}
Here $F_1$ and $F_2$ also depend on the $z$-variable. For
simplicity, we drop this dependence throughout the rest of the section.
Substituting $z=\tilde{z}$ in \eqref{Fgorilla11} and $z=1$ in \eqref{Fgorilla22},
we define
\begin{align*}
  & m[(v_1,\mu,v_2,\lambda);\delta]:=F_1[(v_1,\mu,v_2,\lambda);\delta](\tilde{z}),\\
  & \ell[(v_1,\mu,v_2,\lambda);\delta]:=F_2[(v_1,\mu,v_2,\lambda);\delta](1),
\end{align*}
by substituting $v_1(\tilde{z})=\mu$ and $v_2(1)=\lambda$ in both expressions.

Define the extension of $F_1$ and $F_2$ to $\delta=0$ by using Taylor expansion
together with the symmetry properties of the involved kernels. Let
\[C_1=-S_1''(0)|J_1|-K''(0)|J_2|\,,\quad C_2=-S_2''(0)|J_2|-K''(0)|J_1|\,.\]
Due to the assumptions on $S_1$, $S_2$, and $K$, the constants $S_1''(0)$, $S_2''(0)$ and $K''(0)$ are non-positive.
After some tedious calculations we obtain the natural definition
\begin{align}
& F_1[(v_1,\mu,v_2,\lambda);0](z)=\frac{1}{2}-z-\frac{C_1}{6}v_1(z)^3
+\frac{1}{2}\Big[C_1\mu^2+C_2(\lambda^2-\mu^2)\Big]v_1(z)\,,\label{eq:functional_zero1}\\
& m[(v_1,\mu,v_2,\lambda);0]=\frac{1}{2}-\tilde{z}+\left(\frac{C_1}{3}-\frac{C_2}{2}\right)\mu^3
+\frac{C_2}{2}\mu\lambda^2\,,\nonumber\\
& F_2[(v_1,\mu,v_2,\lambda);0](z)=\tilde{z}-z-\frac{C_2}{6}v_2(z)^3+\frac{C_2}{2}\lambda^2v_2(z)+\frac{C_2}{6}
\mu^3-\frac{C_2}{2}\lambda^2\mu\,,\label{eq:functional_zero2}\\
& \ell[(v_1,\mu,v_2,\lambda);0]=\tilde{z}-1+\frac{C_2}{6}\mu^3+\frac{C_2}{3}\lambda^3 -
\frac{C_2}{2}\mu\lambda^2\,.\nonumber
\end{align}
Our goal is to solve
\[F_1[(v_1,\mu,v_2,\lambda);\delta]=m[(v_1,\mu,v_2\lambda);\delta]
=F_2[(v_1,\mu,v_2,\lambda);\delta]=\ell[(v_1,\mu,v_2,\lambda);\delta]=0\,,\]
for $\delta> 0$ small enough. First we will solve the case $\delta=0$ and then
prove that a solution still exists when $\delta$ is close to zero. The solution
for $\delta=0$ is already partially explicit from formulas
\eqref{eq:functional_zero1}, \eqref{eq:functional_zero2}. We only need to
determine $\mu$ and $\lambda$. To do so, we set the $m$ and $\ell$ components
above equal to zero and obtain
\begin{equation}
\tilde{z}-\frac{1}{2}=\left(\frac{C_1}{3}-\frac{C_2}{2}\right)\mu^3
+\frac{C_2}{2}\mu\lambda^2 \, , \,
1-\tilde{z}=\frac{C_2}{6}\mu^3+\frac{C_2}{3}\lambda^3 -
\frac{C_2}{2}\mu\lambda^2\, , \label{C-formula}
\end{equation}
which should be solved under the constraint $\mu<\lambda$. The second equation in \eqref{C-formula} can be rewritten as
\[(\mu-\lambda)^2(2\lambda+\mu)={6\left(1-\tilde{z}\right)}/{C_2},\]
which describes a cubic hyperbola, asymptotic to the straight lines
$\mu-\lambda=0$ and $2\lambda+\mu=0$ in the $(\mu,\lambda)$ plane.
Consider the branch of
such a curve in the region $0<\mu<\lambda$, intersecting the $\mu=0$ axis at
$(\mu,\lambda)=(0,\lambda_0)$ with \[\lambda_0=\left[{3}\left(1-\tilde{z}
\right)/C_2\right]^{1/3}.\]
Such a branch describes a monotone increasing function $\lambda=\bar{\lambda}(\mu)$ on $\mu>0$, asymptotic to $\lambda=\mu$ as $\mu\rightarrow +\infty$. Now, summing up the two equations in (\ref{C-formula}) we get the following additional condition on $\lambda$ and $\mu$:
\[\lambda = \tilde{\lambda}(\mu):=\left(\frac{3}{2 C_2}+\frac{C_2 - C_1}{C_2}\mu^3\right)^{1/3}.\]
The function $\lambda=\tilde{\lambda}(\mu)$ is also monotone, and attains the value
\[\tilde{\lambda}(0)=\lambda_1 := \left(\frac{3}{2C_2}\right)^{1/3}> \lambda_0
\, . \]
On the other hand, $\tilde{\lambda}(\mu)$ is asymptotic to
the straight line $\lambda = ((C_2-C_1)/C_2)^{1/3}\mu$ as $\mu\rightarrow +\infty$. Since $((C_2-C_1)/C_2)^{1/3}<1$, the two curves $\lambda=\bar{\lambda}(\mu)$ and $\lambda=\tilde{\lambda}(\mu)$ intersect at exactly one point in the region $0<\mu<\lambda$. Hence, $\mu$ and $\lambda$ are uniquely determined.

At this stage, the solution $(v_1,v_2)$ can be easily recovered as the pseudo-inverse variables associated to the densities $\tilde{\rho}_1$ and $\tilde{\rho}_2$
\begin{align*}
& \tilde{\rho}_1(x)=\left[\frac{1}{2}\left(C_1\mu^2+C_2(\lambda^2-\mu^2)\right)-\frac{C_1}{2}x^2\right]_+\mathbf{1}_{[-\mu,\mu]}\,,\\
& \tilde{\rho}_2(x)=\frac{C_2}{2}(\lambda^2-x^2)_+\mathbf{1}_{[-\lambda,-\mu]\cup[\mu,\lambda]}\,,
\end{align*}
which corresponds to two Barenblatt profiles centered at $x=0$ (with possibly different masses and supports) such that the resulting profile $w=\tilde{\rho_1}+\tilde{\rho_2}$ is continuous at $x=\pm \mu$.
This proves that
\[F_1[(v_1,\mu,v_2,\lambda);0] = m[(v_1,\mu,v_2,\lambda);0]=F_2[(v_1,\mu,v_2,\lambda);0]=\ell[(v_1,\mu,v_2,\lambda);0]=0\]
has a unique solution $(v_1,\mu,v_2,\lambda)\in \Omega$. From now on we call
such solution $(v_{0,1},\mu_0,v_{0,2},\lambda_0)$.

\begin{oss}[Inner and outer species]
  \emph{In the above computation we arbitrarily assigned $\rho_1$
to the role of the `inner' species in the segregated state. All of the above
procedure also works with $\rho_2$ as inner species, and therefore we should
 speak of two possible segregated states rather than just one. The only
parameters characterizing each species with respect to this pattern are
$S_1''(0)$ and $S_2''(0)$, which model self-attraction forces. These constants
affect the constants $C_1$ and $C_2$ above, which are non-negative.
When $C_2<C_1$, the self-attraction force of the first species is stronger
than that of the second one, which suggests that $\rho_1$ will concentrate
faster than $\rho_2$, thus producing a pattern in which the first species
occupies the inner region whereas the second species occupies the outer
region. On the other hand, a steady state with the reversed order can still
be constructed under the condition $C_2<C_1$, but we conjecture it to be unstable. This is supported by
numerical simulations shown in Section \ref{sec:simulations}, in which we see that the inner species is the one featuring the `more concentrated' Barenblatt-type profile. We observe that the slope of the function $\lambda=\tilde\lambda(\mu)$ above is negative if $C_2<C_1$ and positive if $C_1<C_2$. This implies that the values of $\mu$ and $\lambda$ solving \eqref{C-formula} are smaller when $C_2<C_1$ compared to the case $C_2>C_1$. Hence, $\rho_1$ and $\rho_2$ have smaller support when $C_2<C_1$. We deduce that the `correct' steady state is `more concentrated' than the unstable one.}
\end{oss}

\begin{oss}[Non-symmetric segregation]
\emph{As shown by the numerical simulations in Section \ref{sec:simulations}, segregation may emerge via a non-symmetric structure, in which $\rho_1+\rho_2$ is supported on a connected interval and $\rho_1$ and $\rho_2$ feature exactly one jump discontinuity, see figure \ref{fig:2} below. This is typically the case for instance when the two species are initially separated. The mathematical proof of the existence of non-symmetric segregated steady states can be carried out similarly to the symmetric case. We omit the details and restrict to the symmetric case for simplicity.}
\end{oss}

\subsection{Solution via implicit function theorem for $\delta>0$}\label{sec:implicit}

In this section we adapt the strategy of \cite[Section 4]{budif} to our two-species problem. In order to solve problem \eqref{gorilla11}, \eqref{gorilla22} for $\delta>0$ small enough, we analyse the operator
\[U_{1/2}\ni(v_1,\mu,v_2,\lambda)\mapsto\mathcal{F}[(v_1,\mu,v_2,\lambda);\delta]:=(F_1,m,F_2,\ell)[(v_1,\mu,v_2,\lambda)]\in\Omega_1.\]

Our first goal is to prove that $\mathcal{F}$ is Frech\'{e}t continuously
differentiable in a neighborhood of $(v_{1,0},\mu_0,v_{2,0},\lambda_0)$, the
solution of $\mathcal{F}[(v_1,\mu,v_2,\lambda);0]=0$. By Taylor expansion
in \eqref{gorilla11}, \eqref{gorilla22} w.r.t. $\delta$ around
$\delta=0$ one sees that $\mathcal{F}$ has a continuous first derivative with respect to $\delta$. We omit the details.

For fixed $\delta>0$, the Jacobian of $\mathcal{F}[\cdot;\delta]$ w.r.t. the first four variables has the block structure
\[D\mathcal{F}=\begin{pmatrix} \frac{\partial F_1}{\partial v_1} & \frac{\partial F_1}{\partial \mu} & \frac{\partial F_1}{\partial v_2} & \frac{\partial F_1}{\partial \lambda}\\
\frac{\partial m}{\partial v_1} & \frac{\partial m}{\partial \mu} & \frac{\partial m}{\partial v_2} & \frac{\partial m}{\partial \lambda}\\
\frac{\partial F_2}{\partial v_1} & \frac{\partial F_2}{\partial \mu} & \frac{\partial F_2}{\partial v_2} & \frac{\partial F_2}{\partial \lambda}\\
\frac{\partial \ell}{\partial v_1} & \frac{\partial \ell}{\partial \mu} & \frac{\partial \ell}{\partial v_2} & \frac{\partial \ell}{\partial \lambda}
\end{pmatrix},\]
where partial derivatives with respect to $v_1$ and $v_2$ are meant to be Frech\'{e}t derivatives. We now compute such terms. Consider perturbations
$  (w_1,a,w_2,b)\in \Omega_{1/2} $
such that
\begin{equation}\label{eq:pert3}
(v_{0,1},\mu_0,v_{0,2},\lambda_0) + (w_1,a,w_2,b) \in U_{1/2}\,.
\end{equation}
We notice that \eqref{eq:pert3} is satisfied if $\interleave (w_1,a,w_2,b)
\interleave_{1/2}$ is small enough, since $v_{0,1}$ and $v_{0,2}$ have their
gradient bounded from below by a positive constant.
Upon extending $w_1$ and $w_2$ odd to $[1-\tilde{z},\tilde{z}]$ and $[0,1-\tilde{z}]\cup[\tilde{z},1]$ respectively, one has
\[\int_{J_1}w_1(\zeta)d\zeta=\int_{J_2}w_2(\zeta)d\zeta=0.\]

Now we compute the partial derivatives of $\mathcal{F}$.
For simplicity, we drop the --$_0$ indices to denote the $\delta=0$ state
and avoid indicating the respective interval for the variable $z$.

Let $(w_1,a,w_2,b)\in \Omega_{1/2}$. By extending all involved functions to $[0,1/2]$ in the usual symmetric form, we get for $\delta>0$
\begin{align*}
& \frac{\partial F_1}{\partial v_1}[(v_1,\mu,v_2,\lambda);\delta](w_1) \\
& \quad\quad = \delta^{-2}\left[\int_{J_1} S_1\big(\delta(v_1(z)
-v_1(\zeta))\big)
(w_1(z)-w_1(\zeta))\right.\\
& \quad\quad\quad \quad \quad \quad \quad \quad  \left.
-w_1(z) S_1\big(\delta(\mu-v_1(\zeta))\big)
+\delta v_1(z) S_1'\big(\delta(\mu-v_1(\zeta))\big)
w_1(\zeta) \, d\zeta\right.\\
& \quad\quad\quad\quad \quad  \left.+\int_{J_2}
K\big(\delta(v_1(z)-v_2(\zeta))\big)w_1(z)
-w_1(z)K\big(\delta(\mu-v_2(\zeta))\big) d\zeta \right]\\
& \quad \quad\quad  +\delta^{-2}w_1(z)\left[\int_{J_2}
S_2\big(\delta(\mu-v_2(\zeta))\big)
-S_2\big(\delta(\lambda-v_2(\zeta))\big) d\zeta\right.\\
& \quad \quad \quad \quad \quad \quad \quad \quad \quad + \left. \int_{J_1}
K\big(\delta(\mu-v_1(\zeta))\big)-K\big(\delta(\lambda-v_1(\zeta))\big)
d\zeta\right]\,,
\end{align*}
and in the limit  $\delta\searrow 0$ we obtain
\begin{equation}
 \frac{\partial F_1}{\partial v_1}[(v_1,\mu,v_2,\lambda);0](w_1) 
= -\frac{C_1}{2}w_1(z)\big(v_1(z)^2-\mu^2\big)
-\frac{C_2}{2} w_1(z)\big(\mu^2-\lambda^2\big).\label{eq:F1v1}
\end{equation}
This limit
is so far just formal.
The same holds for the $\delta\searrow 0$ limits computed below.
However, we shall prove later that these are actually rigorous limits.
Similarly,
\begin{align*}
& \frac{\partial F_1}{\partial v_2}[(v_1,\mu,v_2,\lambda);\delta](w_2) \\
& \quad \quad  = \delta^{-2}\int_{J_2} - K\big(\delta(v_1(z)-v_2(\zeta))\big)
w_2(\zeta)+v_1(z)K'\big(\delta(\mu-v_2(\zeta))\big)w_2(\zeta) d\zeta\\
& \quad \quad\quad + \delta^{-2}v_1(z)\int_{J_2} -S_2'
\big(\delta(\mu-v_2(\zeta))\big)
+S_2'\big(\delta(\lambda-v_2(\zeta))\big) w_2(\zeta) d\zeta.
\end{align*}
One can easily see that, in the $\delta\searrow 0$ limit one gets
\begin{align*}
& \frac{\partial F_1}{\partial v_2}[(v_1,\mu,v_2,\lambda);0](w_2)=0\,.
\end{align*}
We now compute for $\delta>0$,
\begin{align*}
& \frac{\partial F_1}{\partial \mu}[(v_1,\mu,v_2,\lambda);\delta](a) \\
& \quad \quad = \delta^{-1}v_1(z)\left[\int_{J_1}
S_1'\big(\delta(\mu-v_1(\zeta))\big) d\zeta+\int_{J_2} S_2'
\big(\delta(\mu-v_2(\zeta))\big) d\zeta\right] a,
\end{align*}
and a Taylor expansion arguments shows
\begin{align*}
& \frac{\partial F_1}{\partial \mu}[(v_1,\mu,v_2,\lambda);0](a)
= (C_1-C_2)v_1(z)\mu \, a \, .
\end{align*}
Similarly,
\begin{align*}
& \frac{\partial F_1}{\partial \lambda}[(v_1,\mu,v_2,\lambda);\delta](b) \\
& \quad \quad = \delta^{-2}v_1(z)\left[-\delta \int_{J_2}S_2'
\big (\delta(\lambda-v_2(\zeta))\big)d\zeta-\delta\int_{J_1}K'
\big(\delta(\lambda-v_1(\zeta))\big)d\zeta\right] \, b \, ,
\end{align*}
with the $\delta\searrow 0$ limit
\begin{align*}
&\frac{\partial F_1}{\partial \lambda}[(v_1,\mu,v_2,\lambda);0](b) = C_2v_1(z)\lambda b.
\end{align*}
Turning to the second row of $D\mathcal{F}$ we get
\begin{align*}
& \frac{\partial m}{\partial v_1}[(v_1,\mu,v_2,\lambda);\delta](w_1) \\
& \quad = \delta^{-3}\int_{J_1} \Big[ -\delta S_1 \big (\delta(\mu-v_1(\zeta)) \big)
+\delta^2 \mu S_1'\big (\delta(\mu-v_1(\zeta)) \big )\Big ] w_1(\zeta)d\zeta\\
& \quad \quad  + \delta^{-2}\mu \int_{J_1} \Big[ -\delta K'\big
(\delta(\mu-v_1(\zeta))\big )
+\delta K'\big (\delta(\lambda-v_1(\zeta)) \big )\Big ]  w_1(\zeta)d\zeta,
\end{align*}
and a simple symmetry argument shows that
\begin{align*}
& \frac{\partial m}{\partial v_1}[(v_1,\mu,v_2,\lambda);0](w_1) = 0.
\end{align*}
Next we compute
\begin{align*}
& \frac{\partial m}{\partial \mu}[(v_1,\mu,v_2,\lambda);\delta](a) \\
& \quad  = -\delta^{-1}\mu a \int_{J_1}S_1' \big(\delta(\mu-v_2(\zeta))\big)
d\zeta-\delta^{-1}\mu a\int_{J_2} K'\big(\delta(\mu-v_2(\zeta))\big) d\zeta \\
& \quad \quad +\delta^{-2} a \int_{J_2} S_2\big(\delta(\mu-v_2(\zeta))\big)
-S_2\big (\delta(\lambda-v_2(\zeta))\big ) d\zeta\\
& \quad \quad  + \delta^{-2} a \int_{J_1} K \big (\delta(\mu-v_1(\zeta))\big)
-K\big(\delta(\lambda-v_1(\zeta))\big) d\zeta\\
& \quad \quad + \delta^{-1} \mu a\int_{J_2}S_2'\big (\delta(\mu-v_2(\zeta))\big )d\zeta + \delta^{-1}\mu a\int_{J_1} K'\big (\delta(\mu-v_1(\zeta))\big )d\zeta,
\end{align*}
and a simple computation shows in the $\delta\searrow 0$ limit,
\begin{align*}
& \frac{\partial m}{\partial \mu}[(v_1,\mu,v_2,\lambda);0](a) = \left(C_1\mu^2-\frac{3}{2}C_2\mu^2 +\frac{C_2}{2}\lambda^2\right)a\,.
\end{align*}
Similar to the above, we have
\begin{eqnarray*}
&& \frac{\partial m}{\partial v_2}[(v_1,\mu,v_2,\lambda);\delta](w_2) \\
&& \quad  = \delta^{-3}\int_{J_2}\Big[-\delta K\big(\delta(\mu-v_2(\zeta))\big)
+\delta\mu K'\big(\delta(\mu-v_2(\zeta))\big)\Big] w_2(\zeta) d\zeta\\
&& \quad \quad  + \delta^{-2}\mu\int_{J_2}\Big[-\delta S_2'
\big(\delta(\mu-v_2(\zeta))\big)+
\delta S_2\big(\delta(\lambda-v_2(\zeta))\big)\Big]
w_2(\zeta) d\zeta\, , \\
&& \mbox{with } \frac{\partial m}{\partial v_2}[(v_1,\mu,v_2,\lambda);0](w_2) = 0.
\end{eqnarray*}
We conclude the second row of the Jacobian by computing
\begin{align*}
& \frac{\partial m}{\partial \lambda}[(v_1,\mu,v_2,\lambda);\delta](b) \\
& \quad = -\delta^{-1}\mu b \int_{J_2} S_2'\big(\delta(\lambda-v_2(\zeta))\big) d\zeta -\delta^{-1}\mu b \int_{J_1}K'\big(\delta(\lambda-v_2(\zeta))\big) d\zeta,
\end{align*}
with the $\delta\searrow 0$ limit being
\begin{align*}
& \frac{\partial m}{\partial \lambda}[(v_1,\mu,v_2,\lambda);0](b) = C_2\lambda \mu.
\end{align*}
Now looking at the third row of $D\mathcal{F}$, we compute
\begin{align}
& \frac{\partial F_2}{\partial v_1}[(v_1,\mu,v_2,\lambda);\delta](w_1) \nonumber\\
& \quad  =\delta^{-3}\int_{J_1} - \delta K\big(\delta(v_2(z)-v_1(\zeta))\big )w_1(\zeta)
+\delta^2 v_2(z)K'\big (\delta(\lambda-v_1(\zeta))\big ) w_1(\zeta) d\zeta\nonumber\\
& \quad \quad -\delta^{-3}\int_{J_1} -\delta w_1(\zeta)
K\big (\delta(\mu-v_1(\zeta))\big )+\delta^2 \mu w_1(\zeta) K'\big (\delta(\lambda-v_1(\zeta))\big ) d\zeta,\label{eq:F2v1}
\end{align}
and for $\delta\searrow 0$ this shows
\begin{align*}
& \frac{\partial F_2}{\partial v_1}[(v_1,\mu,v_2,\lambda);0](w_1) = 0.
\end{align*}
Then, we have
\begin{align*}
& \frac{\partial F_2}{\partial \mu}[(v_1,\mu,v_2,\lambda);\delta](a) \\
& \quad = -\delta^{-3}\int_{J_2}\delta S_2\big(\delta(\mu-v_2(\zeta))\big) a
-\delta S_2\big(\delta(\lambda-v_2(\zeta))\big) a \, d\zeta\\
& \quad \quad -\delta^{-3}\int_{J_1} \delta K\big(\delta(\mu-v_1(\zeta))\big)
a-\delta K\big(\delta(\lambda-v_1(\zeta))\big)a \, d\zeta,
\end{align*}
with the $\delta\searrow 0$ limit
\begin{align*}
& \frac{\partial F_2}{\partial \mu}[(v_1,\mu,v_2,\lambda);0](a) = \frac{C_2}{2}(\mu^2-\lambda^2) a\,.
\end{align*}
We continue computing the third row of $D\mathcal{F}$ with
\begin{align*}
& \frac{\partial F_2}{\partial v_2}[(v_1,\mu,v_2,\lambda);\delta](w_2) \\
& \quad = \delta^{-3}\int_{J_2} \delta S_2\big(\delta(v_2(z)-v_2(\zeta))\big)
(w_2(z)-w_2(\zeta))-\delta w_2(z) S_2\big(\delta(\lambda-v_2(\zeta))\big)\\
& \quad \quad \quad \quad \quad +\delta^2 v_2(z)
w_2(\zeta)S_2'\big(\delta(\lambda-v_2(\zeta))\big) \, d\zeta\\
& \quad \quad + \delta^{-3}\int_{J_1} \delta w_2(z)K\big(\delta(v_2(z)-
v_1(\zeta))\big)-\delta w_2(z)K\big(\delta(\lambda-v_1(\zeta))\big) \, d\zeta\\
& \quad \quad - \delta^{-3}\int_{J_2} -\delta w_2(\zeta) S_2
\big(\delta(\mu-v_2(\zeta))\big)+\delta^2\mu w_2(\zeta)S_2'
\big(\delta(\lambda-v_2(\zeta))\big) \, d\zeta,
\end{align*}
and the $\delta\searrow 0$ limit
\begin{align*}
   & \frac{\partial F_2}{\partial v_2}[(v_1,\mu,v_2,\lambda);0](w_2) =\frac{C_2}{2}(\lambda^2-v_2^2(z))w_2(z)\,.
\end{align*}
To conclude the third row of the Jacobian, we have
\begin{align*}
   & \frac{\partial F_2}{\partial \lambda}[(v_1,\mu,v_2,\lambda);\delta](b) \\
   & \quad = -\delta^{-3}\int_{J_2}\delta^2 (v_2(z)-\mu)b S_2'
\big(\delta(\lambda-v_2(\zeta))\big)d\zeta\\
   & \quad \quad  -\delta^{-3}\int_{J_1}\delta^{-2}(v_2(z)-\mu)b K'
\big(\delta(\lambda-v_1(\zeta))\big)d\zeta\,,
\end{align*}
and the $\delta\searrow 0$ limit is
\begin{align*}
   & \frac{\partial F_2}{\partial \lambda}[(v_1,\mu,v_2,\lambda);0](b) =C_2(v_2(z)-\mu)\lambda b\,.
\end{align*}
Finally, we analyse the last row of the Jacobian of $\mathcal{F}$.
\begin{align*}
   & \frac{\partial \ell}{\partial v_1}[(v_1,\mu,v_2,\lambda);\delta](w_1) \\
   & \quad = \delta^{-3}\int_{J_1}\Big[-\delta
K\big(\delta(\lambda-v_1(\zeta))\big)
+ \delta K\big(\delta(\mu-v_1(\zeta))\big)\Big] w_1(\zeta)d\zeta\\
   & \quad \quad -\delta^{-3}\int_{J_1}\Big[-\delta^2 \lambda K'
\big(\delta(\lambda-v_1(\zeta))\big)+\delta^2 \mu K'
\big(\delta(\lambda-v_(\zeta))\big)\Big]w_1(\zeta) d\zeta,
\end{align*}
and the $\delta\searrow 0$ limit can be easily computed to be
\begin{align*}
   & \frac{\partial \ell}{\partial v_1}[(v_1,\mu,v_2,\lambda);0](w_1) =0.
\end{align*}
Then, we have
\begin{align*}
   & \frac{\partial \ell}{\partial \mu}[(v_1,\mu,v_2,\lambda);\delta](a) \\
   & \quad  = \delta^{-3} a \int_{J_2} -\delta S_2\big(\delta(\mu-v_2(\zeta))\big)
+\delta S_2\big(\delta(\lambda-v_2(\zeta))\big) \, d\zeta\\
   & \quad \quad +\delta^{-3}a\int_{J_1} -\delta K\big(\delta(\mu-v_1(\zeta))\big)
+\delta K\big(\delta(\lambda-v_1(\zeta))\big) \, d\zeta,
\end{align*}
with the $\delta\searrow 0$ limit being
\begin{align*}
   & \frac{\partial \ell}{\partial \mu}[(v_1,\mu,v_2,\lambda);\delta](a) =\frac{C_2}{2}(\mu^2-\lambda^2)a.
\end{align*}
We then continue with
\begin{align*}
   & \frac{\partial \ell}{\partial v_2}[(v_1,\mu,v_2,\lambda);\delta](w_2) \\
   & \quad = \delta^{-3}\int_{J_2} -\delta w_2(\zeta) S_2\big(\delta(\lambda-v_2
(\zeta))\big)+\delta w_2(\zeta)S_2\big(\delta(\mu-v_2(\zeta))\big) \, d\zeta\\
   & \quad \quad + \delta^{-3}\int_{J_2} \delta^2 \lambda w_2(\zeta)S_2'
\big(\delta(\lambda-v_2(\zeta))\big)-\delta^2\mu w_2(\zeta)S_2'
\big(\delta(\lambda-v_2
(\zeta))\big) \, d\zeta,
\end{align*}
and
\begin{align*}
   & \frac{\partial \ell}{\partial v_2}[(v_1,\mu,v_2,\lambda);0](w_2) =0.
\end{align*}
The last term is
\begin{align*}
   & \frac{\partial \ell}{\partial \lambda}[(v_1,\mu,v_2,\lambda);\delta](b) \\
   & \quad = \delta^{-3}b\int_{J_2} -\delta^2\lambda S_2'\big(\delta(\lambda-
v_2(\zeta))\big)+\delta^2\mu S_2'\big(\delta(\lambda-v_2(\zeta))\big)
\, d\zeta\\
   & \quad \quad +\delta^{-3}b\int_{J_1} -\delta^2 \lambda K'
\big(\delta(\lambda-
v_1(\zeta))\big)+\delta^2 \mu K'\big(\delta(\lambda-v_1(\zeta))\big) \, d\zeta,
\end{align*}
and the $\delta\searrow 0$ limit is
\begin{align*}
   & \frac{\partial \ell}{\partial \lambda}[(v_1,\mu,v_2,\lambda);0](b) =C_2(\lambda^2-\mu\lambda)b\,.
\end{align*}

The above computations show for $\delta$ small enough, that
$D\mathcal{F}[(v_{0,1},\mu_0,v_{0,2},\lambda_0);\delta]$ is a bounded linear
operator from $\Omega$ into itself, and that $D\mathcal{F}$ is continuous at
$\delta=0$ in the operator norm.
This is easily seen via Taylor expansion, using bounds on the $L^\infty$ norms,
 and symmetry properties.

\begin{lem}\label{lem1}
$D\mathcal{F}[(v_{0,1},\mu_0,v_{0,2},\lambda_0);\delta]$ is a bounded linear operator from $\Omega_{1/2}$ to $\Omega_1$ for $\delta>0$ small enough.
\end{lem}

\proof
Due to the structure of the spaces $\Omega_\alpha$, we only need to check the
following. Define for $z\in J_2$ and $\delta\geq 0$,
\begin{eqnarray*}
   g_2(\delta;z)&=& \frac{\partial F_2}{\partial v_1}[(v_1,\mu,v_2,\lambda);\delta](w_1)+ \frac{\partial F_2}{\partial \mu}[(v_1,\mu,v_2,\lambda);\delta](a) \\
  &&  +\frac{\partial F_2}{\partial v_2}[(v_1,\mu,v_2,\lambda);\delta](w_2)+\frac{\partial F_2}{\partial \lambda}[(v_1,\mu,v_2,\lambda);\delta](b) \\
  && \ - \frac{\partial \ell}{\partial v_1} [(v_1,\mu,v_2,\lambda);\delta](w_1)- \frac{\partial \ell}{\partial \mu} [(v_1,\mu,v_2,\lambda);\delta](a) \\
  &&  \ -
  \frac{\partial \ell}{\partial v_2} [(v_1,\mu,v_2,\lambda);\delta](w_2)
-\frac{\partial \ell}{\partial \lambda} [(v_1,\mu,v_2,\lambda);\delta](b) \, .
\end{eqnarray*}
We need to prove that
\[\sup_{\interleave (w_1,a,w_2,b)\interleave_{1/2}\leq 1}\frac{1}{1-z}
\big|g_2(\delta;z)-g_2(0;z)\big|\searrow 0,\qquad \hbox{as $ \delta\searrow 0$}.\]
Consider all the above terms separately, and for notational
reasons omit the dependence on $(v_1,\mu,v_2,\lambda)$. By Taylor expansion,
and using simple symmetry properties, we get
\begin{align}
  & \frac{1}{1-z}\left[\left(\frac{\partial F_2}{\partial v_1}(\delta)
-\frac{\partial F_2}{\partial v_1}(0)\right)(w_1)
-\left(\frac{\partial \ell}{\partial v_1}(\delta)-\frac{\partial \ell}
{\partial v_1}(0)\right)(w_1) \right]\nonumber\\
  & \quad = -\delta\int_{J_1}w_1(\zeta)\left[\frac{K''(\bar{x}(\zeta))}{2}\frac{(\lambda-v_1(z))^2}{1-z}+\frac{K'''(\tilde{x}(\zeta))}{6}\frac{(v_2(z)-\lambda)^3}{1-z}\right]
  d\zeta,\label{eq:derivative1}
\end{align}
for some $\bar{x}(\zeta)$ and $\tilde{x}(\zeta)$ in the interval
$[0,\lambda\delta]$. Now, it can be proved (cf. a similar argument in \cite[Lemma 4.1]{budif}) that $(\lambda-v_2(z))/(1-z)^{1/2}$ is uniformly bounded in
$[\tilde{z},1]$. This provides the desired estimate for the related term
in \eqref{eq:derivative1}.
One can check that
\[\frac{\partial F_2}{\partial \mu}(\delta)-\frac{\partial \ell}{\partial \mu}(\delta)-\frac{\partial F_2}{\partial \mu}(0)+\frac{\partial \ell}{\partial \mu}(0)\equiv 0.\]
We now estimate
\begin{align*}
    & \frac{1}{1-z}\left[\left(\frac{\partial F_2}{\partial v_2}(\delta)
-\frac{\partial F_2}{\partial v_2}(0)\right)(w_2) +\left(\frac{\partial F_2}
{\partial \lambda}(\delta)-\frac{\partial F_2}{\partial \lambda}(0)\right)(b)
\right.\\
    & \quad \quad \quad \quad \left. -\left(\frac{\partial \ell}{\partial v_2}(\delta)
-\frac{\partial \ell}{\partial v_2}(0)\right)(w_2) - \left(\frac{\partial \ell}
{\partial \lambda}(\delta)-\frac{\partial \ell}{\partial \lambda}(0)\right)(b)
\right]\\
    & \quad =  \frac{1}{1-z}\left\{\int_{J_2}w_2(\zeta)
\left[\frac{1}{2}(v_2(z)-\lambda)^2 S_2''\big(\delta(\lambda-v_2(\zeta))\big)
+\frac{\delta}{6}(\lambda-v_2(z))^3 S_2'''(\bar{x})\right]d\zeta\right.\\
    &\quad \quad \quad \quad \quad  + w_2(z)\int_{J_2} \delta^{-2}
S_2\big(\delta(v_2(z)-v_2(\zeta))\big)
-S_2\big(\delta(\lambda-v_2(\zeta))\big) \, d\zeta\\
    & \quad \quad \quad \quad \quad + w_2(z)\int_{J_1} \delta^{-2}
K\big(\delta(v_2(z)-v_1(\zeta))\big)
-K\big(\delta(\lambda-v_1(\zeta))\big) \, d\zeta\\
    & \quad \quad \quad \quad \quad  + b\delta^{-1}(\lambda-v_2(z))
\left[\int_{J_2}S_2'\big(\delta(\lambda-v_2(\zeta))\big)d\zeta
+\int_{J_1}K'\big(\delta(\lambda-v_1(\eta))\big)d\zeta\right]\\
    & \quad \quad \quad \quad \quad  \left.-C_2\lambda b (v_2(z)-\lambda)+\frac{C_2}{2}w_2(z)(v_2(z)^2
-\lambda^2)\right \},
\end{align*}
and some tedious Taylor expansions imply
\begin{align*}
     &  \frac{1}{1-z}\left[\left(\frac{\partial F_2}{\partial v_2}(\delta)
-\frac{\partial F_2}{\partial v_2}(0)\right)(w_2) + \left(\frac{\partial F_2}
{\partial \lambda}(\delta)-\frac{\partial F_2}{\partial \lambda}(0)\right)(b)
\right.\\
    & \quad \quad \quad \left. -\left(\frac{\partial \ell}{\partial v_2}(\delta)
-\frac{\partial \ell}{\partial v_2}(0)\right)(w_2) - \left(\frac{\partial \ell}
{\partial \lambda}(\delta)-\frac{\partial \ell}{\partial \lambda}(0)\right)(b)
\right]\\
    & \quad =  O(\delta^2)(1+w_2(z))\left(\frac{(\lambda-v_2(z))^2}{1-z}+ \frac{(\lambda-v_2(z))^4}{1-z}\right)\\
    & \quad \quad  + \frac{(w_2(z)-b)(v_2(z)-\lambda)}{1-z}O(\delta).
\end{align*}
This proves the assertion.
\endproof

\begin{lem}\label{lem2}
$D\mathcal{F}[(v_{0,1},\mu_0,v_{0,2},\lambda_0);0]$ is a linear isomorphism
between $\Omega_{1/2}$ and $\Omega_1$
for $\delta>0$ small enough.
\end{lem}
\proof
We observe that the Jacobian of $\mathcal{F}$ at $\delta=0$ has the structure
\[D\mathcal{F}[(v_1,\mu,v_2,\lambda);0]=
\begin{pmatrix}
\frac{\partial F_1}{\partial v_1} & (C_1-C_2)v_1 \mu & 0 & C_2 v_1\mu \\
0 & C_1\mu^2 -\frac{3}{2}C_2\mu^2+\frac{C_2}{2}\lambda^2 & 0 & C_2\lambda\mu \\
0 & \frac{C_2}{2}(\mu^2-\lambda^2) & \frac{\partial F_2}{\partial v_2} & C_2(v_2-\mu)\lambda \\
0 & \frac{C_2}{2}(\mu^2-\lambda^2) & 0 & C_2(\lambda^2-\mu\lambda)
\end{pmatrix}.\]
Given $(h_1,\alpha,h_2,\beta)\in \Omega_{1}$, we have to prove that
\begin{equation}\label{eq:linearsystem}
  (w_1,a,w_2,b)^T=D\mathcal{F}[(v_1,\mu,v_2,\lambda);0](h_1,\alpha,h_2,\beta)^T,
\end{equation}
admits a unique solution $(w_1,a,w_2,b)\in \Omega_{1/2}$ with
\[\interleave (w_1,a,w_2,b)\interleave_{1/2}\leq C\interleave (h_1,\alpha,h_2,\beta)\interleave_{1},\]
for some $C>0$ independent of $(h_1,\alpha,h_2,\beta)$.

As a first step, we claim that $\frac{\partial F_1}{\partial v_1}$ is
invertible as a map from $L^\infty$ to $L^\infty$ at $\delta=0$. To see this, we use
\eqref{eq:F1v1}. For $h_1(z)=\frac{\partial F_1}{\partial v_1}(w_1)$ we get
\[\|w_1\|_{L^\infty(J_1)}= 2 \left\|\Big (C_1(v_1(z)^2-\mu^2)+C_2(\mu^2
-\lambda^2)\Big)^{-1}\right\|_{L^\infty(J_1)}\|h_1\|_{L^\infty(J_1)},\]
and the assertion follows. Therefore, the proof will be completed if we can
show that the sub-matrix
\[
\begin{pmatrix}
C_1\mu^2 -\frac{3}{2}C_2\mu^2+\frac{C_2}{2}\lambda^2 & 0 & C_2\lambda\mu \\
\frac{C_2}{2}(\mu^2-\lambda^2) & \frac{\partial F_2}{\partial v_2} & C_2(v_2-\mu)\lambda \\
\frac{C_2}{2}(\mu^2-\lambda^2) & 0 & C_2(\lambda^2-\mu\lambda)
\end{pmatrix},
\]
is an invertible operator in the components $(a,w_2,b)$. First, we prove that
\begin{equation}\label{eq:condition1}
  C_2\lambda(\lambda-\mu)\left(C_1\mu^2-C_2\mu^2+\frac{C_2}{2}\lambda^2+\frac{C_2}{2}\mu\lambda\right)\neq 0,
\end{equation}
which is equivalent to
\[ C_2\lambda(\lambda-\mu)\left(C_1\mu^2+\frac{C_2}{2}(\lambda-\mu)(\lambda+2\mu)\right)\neq 0.\]
This is always satisfied since $\lambda>\mu$. Condition \eqref{eq:condition1} implies that the linear system
\begin{eqnarray*}
   \left[C_1\mu^2 -\frac{3}{2}C_2\mu^2+\frac{C_2}{2}\lambda^2\right] a
+ C_2\lambda\mu b &=& \alpha, \\
   \frac{C_2}{2}(\mu^2-\lambda^2)a+ C_2(\lambda^2-\mu\lambda)b &=& \beta,
\end{eqnarray*}
has a unique solution $(a,b)$. Now we only need to determine $w_2$. By subtracting the last two rows of the linear system \eqref{eq:linearsystem}, and by some simple manipulation, we obtain
\[w_2(z)-b=\frac{2}{\lambda+v_2(z)}\left[\frac{b}{2}(\lambda-v_2(z))+\frac{2(h_2(z)-\beta)}{C_2(\lambda-v_2(z))}\right].\]
Since
$(\lambda-v_2(z))/(1-z)$
is uniformly positive on $[\tilde{z},1]$ (cf. a similar proof in
\cite[Lemma 1.4]{budif}), we obtain the desired assertion, by dividing the
above identity by $\sqrt{1-z}$.
\endproof

We are now ready to prove the main theorem of this section, Theorem \ref{teo_main}, as well as one of the most important results in this paper.


\begin{proof}[Proof of Theorem \ref{teo_main}]
The results in this section, in particular Lemma \ref{lem1} and Lemma \ref{lem2},
together with the implicit function theorem on Banach spaces (see e. g.
\cite[Theorem 15.1]{deimling}), imply that the functional equation
$\mathcal{F}[(v_1,\mu,v_2,\lambda);\delta]=0$ has a solution for $\delta$ small
enough. Here
$\mathcal{F}$ is defined in Subsection \ref{subsec:functional} (see in particular \eqref{Fgorilla11}) and at the
beginning of Subsection \ref{sec:implicit}. Hence, we obtain a solution $(v_1,v_2)$ to \eqref{gorilla11}-\eqref{gorilla22}.
The computations in Subsection \ref{subsec:formulation} imply that $(u_1,u_2)$
with $u_i=\delta v_i$, $i=1,2$, is a solution to \eqref{pseudo_stat} once we
achieve enough regularity for the $v_i$. This follows easily from \eqref{gorilla11}-\eqref{gorilla22}. Indeed, since $\delta$ is very small and since
$S_1$, $S_2$, and $K_0$ are bounded away from zero, we can easily recover
the $\partial_z v_i$ after differentiating \eqref{gorilla11}-\eqref{gorilla22}
with respect to $z$. In particular, we find that the $\partial_z v_i$ are
bounded away from zero. Hence, we can divide the resulting equations by
$\partial_z v_1$ and $\partial_z v_2$ respectively. Similar to the above, we can once again differentiate (using the $C^1$ regularity of the $v_i$) and
obtain \eqref{pseudo_stat} for $(u_1,u_2)$. The usual change of variable
transforming pseudo inverse variables to densities shows that $(\rho_1,\rho_2)$ solves \eqref{eq.stationary} where $\rho_i=\partial_x F_i$, $i=1,2$, and $F_i$ is the pseudo-inverse of $u_i$ for $i=1,2$.
\end{proof}

\section{Numerical Simulations}\label{sec:simulations}

Here we present some examples of sorting phenomena and mixing by solving
\eqref{eq.diffusion_1} numerically in one space dimension. We use a particle
method introduced for equations in gradient flow form in
\cite{car_pat1,car_pat2}, which is equivalent to a finite difference scheme for the pseudo-inverse equation, see also \cite{gosse_toscani} and \cite{dffrr} for scalar conservation laws.

In this section we denote by $\rho$ and $\eta$ the two densities and by $u$
and $v$ the corresponding pseudo-inverse functions, that are solutions of the following system
\begin{equation*}
\begin{cases}
 \frac{du}{dt}= -\epsilon\frac12\left(\left(u_z\right)^{-2}\right)_z-\epsilon\eta_x(u)+\int_0^1\left(S_1'(u(z)-u(\zeta))+K'(u(z)-v(\zeta))\right)d\zeta,		\\
 \frac{dv}{dt}= -\epsilon\frac12\left(\left(v_z\right)^{-2}\right)_z-\epsilon\rho_x(v)+\int_0^1\left(S_2'(v(z)-v(\zeta))+K'(v(z)-u(\zeta))\right)d\zeta,	
\end{cases}
\end{equation*}
for $z\in\left[0,1\right]$.
This is in order to avoid confusion w.r.t. the indices for discretization.
Clearly, the masses of $\rho$ and $\eta$ are normalized to one. The main
issue for the equations above is how to treat the cross-diffusion part
numerically. We intentionally left the cross-diffusion terms above in the
form of `external potentials'. Given $N\in\mathbb{N}$ we consider a partition
of the interval $\left[0,1\right]$, $\left\{z_i\right\}_{i=1}^N$ and call $u(z_i)=u_i$, $v(z_i)=v_i$, and $m=\frac{1}{N}$. The discretization in space
the reads
\begin{align*}
 & d_t u_i= \frac{m}{2}\frac{(u_{i+1}-u_i)^2-(u_{i}-u_{i-1})^2}{(u_{i+1}-u_i)^2}-\eta_x(u_i)+m\sum_{j=1}^N\left(S_1'(u_i-u_j)+K'(u_i-v_j)\right),		\\
 & d_t v_i= \frac{m}{2}\frac{(v_{i+1}-v_i)^2-(v_{i}-v_{i-1})^2}{(v_{i+1}-v_i)^2}-\rho_x(v_i)+m\sum_{j=1}^N\left(S_2'(v_i-v_j)+K'(v_i-u_j)\right).		
\end{align*}
for $i=1,...,N$. Integrating the above ODE system we get two families of particles $\left\{u_i(t)\right\}_{i=1}^N$, $\left\{v_i(t)\right\}_{i=1}^N$, for $t\in\left[0,T\right]$ and reconstruct the discrete densities as follows
\begin{align*}
&  \rho^N(t,x)=\sum_{i=1}^N\frac{2m}{u_{i+1}(t)-u_{i-1}(t)}\mathbf{1}_{\left\{\left(u_{i-\frac12}(t),u_{i+\frac12}(t)\right)\right\}},\\
&  \eta^N(t,x)=\sum_{i=1}^N\frac{2m}{v_{i+1}(t)-v_{i-1}(t)}\mathbf{1}_{\left\{\left(v_{i-\frac12}(t),v_{i+\frac12}(t)\right)\right\}}.
\end{align*}
Given the two initial conditions $\rho^0$ and $\eta^0$ the initial positions for the ODE system are determined via the atomization
\begin{align*}
  & u_1^0=  \sup_{x\in\R}\left\{\int_{-\infty}^x \rho^0(y)dy<\frac1N\right\}
\quad , \quad  u_i^0=  \sup_{x\in\R}\left\{\int_{u_{i-1}^0}^x \rho^0(y)dy<\frac1N\right\},\\
  & v_1^0=  \sup_{x\in\R}\left\{\int_{-\infty}^x \eta^0(y)dy<\frac1N\right\}
\quad , \quad  v_i^0=  \sup_{x\in\R}\left\{\int_{v_{i-1}^0}^x \eta^0(y)dy<\frac1N\right\},
\end{align*}
for $i=2,...,N$.
The cross-diffusion part is reconstructed at each time iteration using the
discretized density. For the nonlocal part we choose
%
%
\[
 S_i=\sigma_i K \  , \   \sigma_i > 0 \  , \  i=1,2 \ , \mbox{ and }
\sigma_1 + \sigma_2 > 2 \ ,
\]
where $K$ is a normalized Gaussian potential. For a diffusion coefficient
$\epsilon=1$, we show segregation phenomena in Figures \ref{fig:1},
\ref{fig:2} for two different choices of initial conditions.
Two different types of segregation are possible. In \figurename~\ref{fig:1} the initial data for the two species are perfectly matching, this produces symmetric segregated states in the large time limit. In \figurename~\ref{fig:2} the two initial data are shifted, this produces a non symmetric segregation in which the two species form two adjacent patterns with connected support. Mixing is shown
in Figures \ref{fig:3}, \ref{fig:4} for the diffusion dominated regime, namely
$\sigma_1+\sigma_2<2$ and diffusion coefficient $\epsilon>1$. Again, different
situations may arise depending on the initial data. In the former case
(perfectly overlapping initial conditions) the two species are almost entirely
overlapping for large times, whereas in the latter case they
overlap in a proper subset of the support of $\rho+\eta$. In all simulations
we have $N=50$ and final time $T=2$.
\begin{figure}[htbp]
\includegraphics[width=9cm]{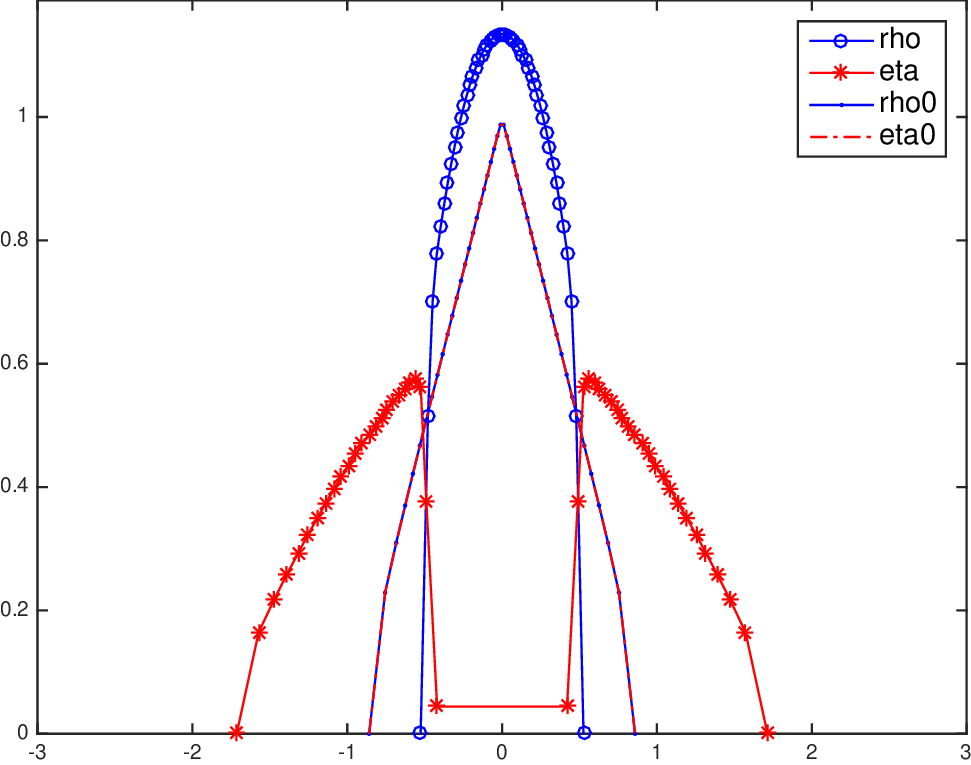}
\caption{Symmetric segregation for $\epsilon=1$, $\sigma_1=10$, $\sigma_2=1.5$, $\rho_0(x)=\eta_0(x)=\left(1-|x|\right)_+$}\label{fig:1}
\end{figure}

\begin{figure}[htbp]
\includegraphics[width=9cm]{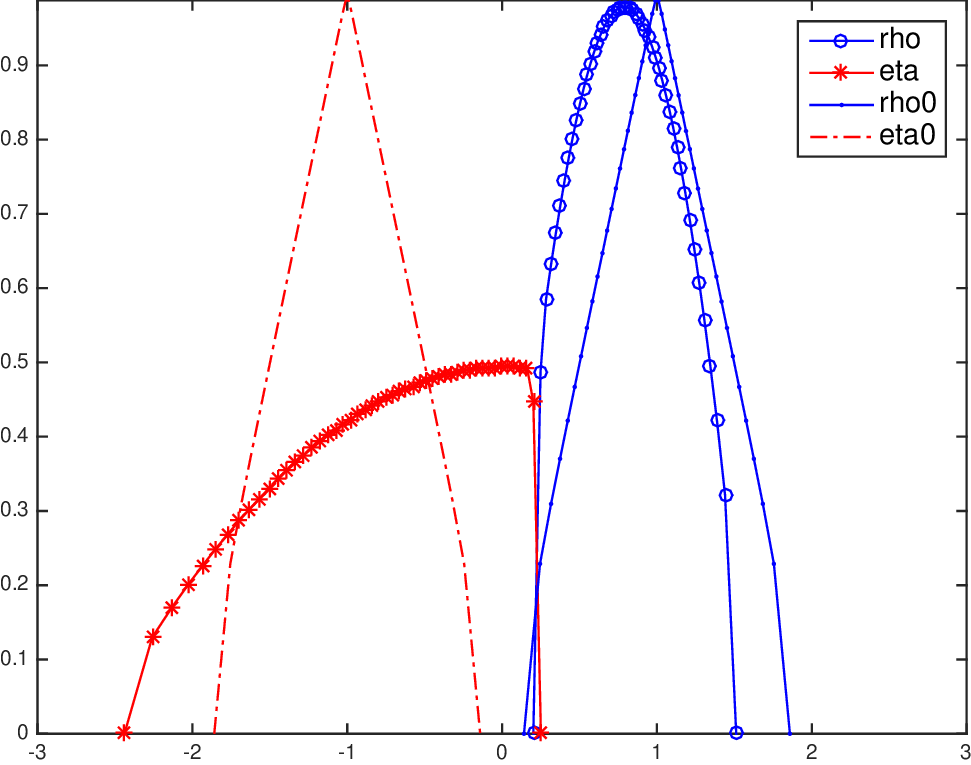}
\caption{Segregation for two densities that are initially disjointed, with $\epsilon=1$, $\sigma_1=10$, $\sigma_2=1.5$, $\rho_0(x)=\left(1-|x-1|\right)_+$, $\eta_0(x)=\left(1-|x+1|\right)_+$}\label{fig:2}
\end{figure}

\begin{figure}[htbp]
\includegraphics[width=9cm]{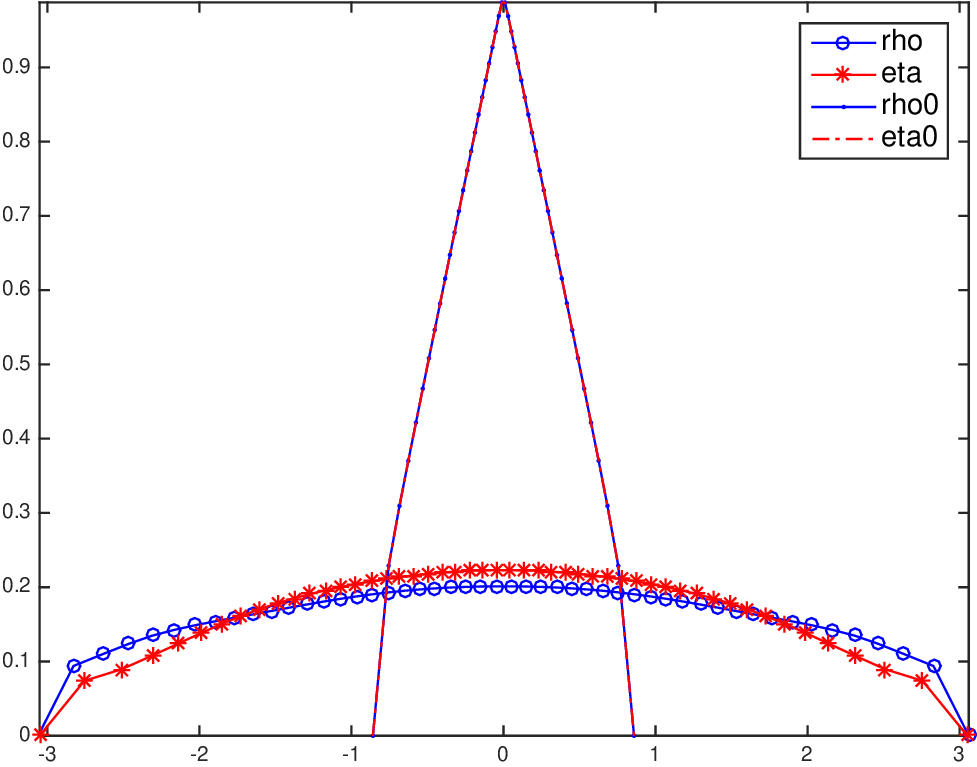}
\caption{Diffusion dominated regime for $\epsilon=3$, $\sigma_1=0.1$, $\sigma_2=0.8$, $\rho_0(x)=\eta_0=(x)=\left(1-|x|\right)_+$}\label{fig:3}
\end{figure}

\begin{figure}[htbp]
\includegraphics[width=9cm]{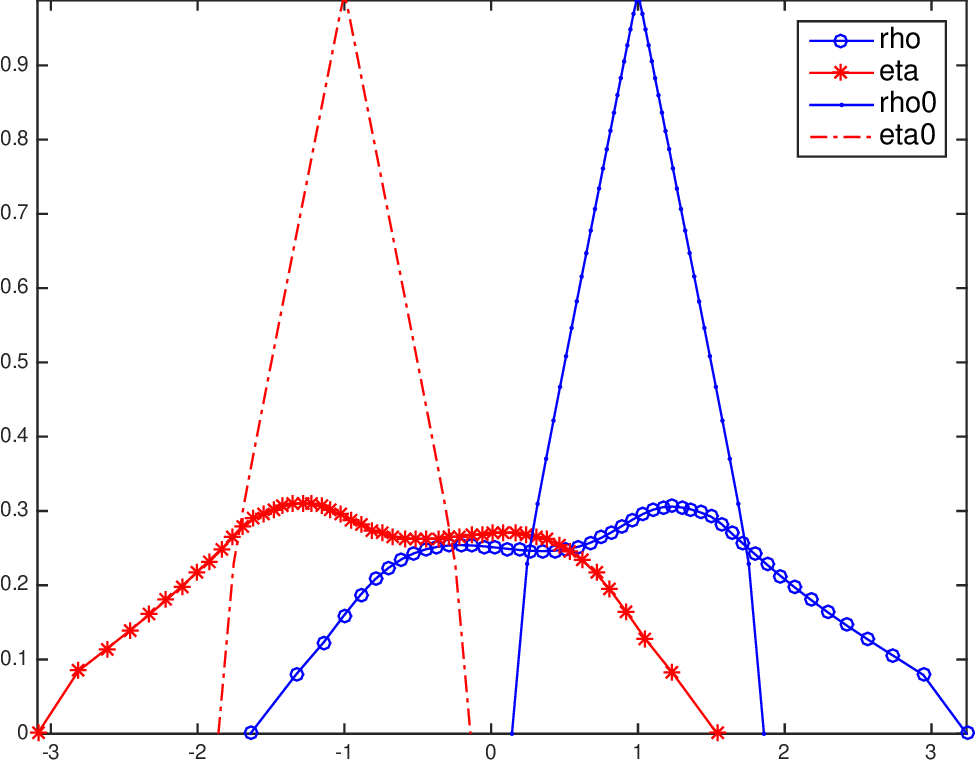}
\caption{Mixing phenomena in the diffusion dominated regime. Here $\epsilon=3$, $\sigma_1=0.1$, $\sigma_2=0.8$, $\rho_0(x)=\left(1-|x-1|\right)_+$, $\eta_0(x)=\left(1-|x+1|\right)_+$}\label{fig:4}
\end{figure}

\section*{Acknowledgments}

MDF and SF were supported by the EU Erasmus Mundus MathMods programme, www.mathmods.eu, by the GNAMPA (Italian group of Analysis, Probability,
and Applications) projects \emph{Geometric and qualitative properties of solutions
to elliptic and parabolic equations}, respectively \emph{Analisi e
stabilit\`a per modelli di
equazioni alle derivate parziali nella matematica applicata}. MDF and
SF acknowledge hospitality by the Institut f\"ur Angewandte Mathematik: Analysis
und Numerik, Westf\"alische Wilhelms-Universit\"at M\"unster, which hosted them
in 2012 when the first ideas of this work came out. This work was done while
 MB and AS were PIs of the Excellence Cluster {\sl{Cells in Motion}} (CiM), funded
by the DFG.


\end{document}